\documentclass{amsart}

\usepackage{hyperref, amssymb, amsmath, amscd}
\usepackage[all]{xy}
\usepackage{paralist, enumerate}

\allowdisplaybreaks[4]

\theoremstyle{plain}
\newtheorem{theorem}{Theorem}[section]
\newtheorem{proposition}[theorem]{Proposition}
\newtheorem{lemma}[theorem]{Lemma}

\theoremstyle{definition}
\newtheorem{remark}[theorem]{Remark}
\newtheorem{notation}[theorem]{Notation}

\newcommand{\CC}{\mathbb{C}}
\newcommand{\FF}{\mathbb{F}}
\newcommand{\GG}{\mathbb{G}}
\newcommand{\PP}{\mathbb{P}}
\newcommand{\ZZ}{\mathbb{Z}}

\newcommand{\LLambda}{\boldsymbol{\Lambda}}

\newcommand{\F}{\mathcal{F}}
\newcommand{\G}{\mathcal{G}}
\newcommand{\I}{\mathcal{I}}
\newcommand{\J}{\mathcal{J}}
\newcommand{\K}{\mathcal{K}}
\newcommand{\Osh}{\mathcal{O}}
\newcommand{\Z}{\mathcal{Z}}
\newcommand{\Extsh}{{\mathcal Ext}}
\newcommand{\Kersh}{{\mathcal Ker}}

\newcommand{\Bf}{\mathfrak{B}}
\newcommand{\If}{\mathfrak{I}}
\newcommand{\Nf}{\mathfrak{N}}

\newcommand{\Aut}{\operatorname{Aut}}
\newcommand{\Ann}{\operatorname{Ann}}
\newcommand{\chern}{\operatorname{c}_1^{}}
\newcommand{\End}{\operatorname{End}}
\newcommand{\euler}{\operatorname{\chi}}
\newcommand{\eulertop}{\operatorname{e}}
\newcommand{\ext}{\operatorname{ext}}
\newcommand{\Ext}{\operatorname{Ext}}
\def\H{\operatorname{H}}
\newcommand{\Hom}{\operatorname{Hom}}
\newcommand{\id}{\operatorname{id}}
\newcommand{\Img}{\operatorname{Im}}
\newcommand{\incl}{\operatorname{in}}
\newcommand{\Grot}{\operatorname{K}}
\newcommand{\Ker}{\operatorname{Ker}}
\newcommand{\length}{\operatorname{length}}
\newcommand{\mult}{\operatorname{r}}
\newcommand{\p}{\operatorname{p}}
\newcommand{\Part}{\operatorname{\mathcal{P}}}
\newcommand{\Poly}{\operatorname{P}\!\!}
\newcommand{\pr}{\operatorname{pr}}
\newcommand{\rank}{\operatorname{rank}}
\newcommand{\thespan}{\operatorname{span}}
\newcommand{\Spec}{\operatorname{Spec}}
\newcommand{\Stab}{\operatorname{Stab}}
\newcommand{\Sym}{\operatorname{Sym}}

\newcommand{\pss}{\mathrm{pss}}
\newcommand{\red}{\mathrm{red}}
\newcommand{\ses}{\mathrm{ss}}
\newcommand{\st}{\mathrm{st}}
\newcommand{\dual}{{\scriptscriptstyle \mathrm{D}}}
\newcommand{\ddual}{{\scriptscriptstyle \mathrm{DD}}}
\newcommand{\Rt}{\mathrm{rt}}
\newcommand{\Lt}{\mathrm{lt}}

\newcommand{\Asing}{\mathrm{A}_{\mathrm{sing}}^{}}
\newcommand{\alphamax}{\alpha_{\mathrm{max}}^{}}
\newcommand{\fiberR}{\mathrm{F}_{\mathrm{rt}}^{}}
\newcommand{\fiberL}{\mathrm{F}_{\mathrm{lt}}^{}}
\newcommand{\GL}{\mathrm{GL}}
\newcommand{\GM}{\mathrm{G}_M^{}}
\newcommand{\Hilb}{\mathrm{Hilb}}
\newcommand{\M}{\mathrm{M}}
\newcommand{\MX}{\mathrm{M}_X^{\alpha}(r, t)}
\newcommand{\MXinf}{\mathrm{M}_X^{\infty}(r, t)}
\newcommand{\MXzero}{\mathrm{M}_X^{0+}(r, t)}
\newcommand{\MXprime}{\mathrm{M}_X^{\alpha}(r', t')^\st}
\newcommand{\MXsecond}{\mathrm{M}_X^{}(r'', t'')}
\newcommand{\pt}{\mathrm{pt}}

\newcommand{\lra}{\longrightarrow}
\newcommand{\tensor}{\otimes}
\newcommand{\rep}[2]{\left( \!\!\! \binom{#1}{#2} \!\!\! \right)}
\newcommand{\Sqcup}{\text{\LARGE $\sqcup$}}

\begin{document}

\setdefaultleftmargin{19pt}{}{}{}{}{}

\title[Variation of moduli spaces of coherent systems of dimension one]
{Variation of moduli spaces of coherent systems of dimension one and order one}

\author{Mario Maican}
\address{Institute of Mathematics of the Romanian Academy, Calea Grivitei 21, Bucharest 010702, Romania}

\email{maican@imar.ro}

\begin{abstract}
We study the wall-crossing for moduli spaces of coherent systems of dimension one and order one
on a smooth projective variety over the complex numbers.
We compute the topological Euler characteristic of the moduli spaces
in the particular case when the variety is a quadric surface,
the first Chern class of the coherent systems is of the form $(2, r)$
and the second Chern class is bounded from below by $3 r + 1$ and also by $4 r - 8$.
\end{abstract}


\maketitle

\section{Introduction}
\label{introduction}

\noindent
Let $X$ be a smooth projective polarized variety over $\CC$.
Let $r > 0$ and $t$ be integers.
Let $\alpha$ be a positive real number.
We denote by $\MX$ the coarse moduli space, constructed by Le~Potier in \cite{lepotier_asterisque},
of S-equivalence classes $\langle \Lambda \rangle$
of $\alpha$-semi-stable coherent systems $\Lambda = (\Gamma, \F)$ on $X$ having order $\dim_\CC^{} \Gamma = 1$
and Hilbert polynomial $\Poly_\F^{}(x) = r x + t$.
We denote by $\M_X^{}(r, t)$ the coarse moduli space, constructed by Simpson in \cite{simpson},
of S-equivalence classes $\langle \F \rangle$ of Gieseker-semi-stable coherent sheaves $\F$ on $X$
having Hilbert polynomial $\Poly_\F^{}(x) = r x + t$.
Let
\[
\MX^\st \subset \MX \qquad \text{and} \qquad \MX^\pss = \MX \smallsetminus \MX^\st
\]
be the open subscheme of isomorphism classes of $\alpha$-stable coherent systems,
respectively, the closed subscheme of properly $\alpha$-semi-stable coherent systems.
We equip the latter with the induced reduced structure.
We say that $\alpha$ is \emph{singular} relative to the polynomial $P(x) = r x + t$ if $\MX^\pss \neq \emptyset$.
Relative to a fixed polynomial of degree $1$, there are finitely many singular values $\alpha_1^{} < \dots < \alpha_m^{}$.
We write $\alpha_0^{} = 0$, $\alpha_{m + 1}^{} = \infty$.
The space $\MX$ remains unchanged as $\alpha$ varies in an interval $(\alpha_i^{}, \alpha_{i+1}^{})$.
At a singular value $\alpha = \alpha_i^{}$ we have the wall-crossing diagram~\eqref{wall-crossing}.
We say that $\Lambda$ is \emph{$0^{+}$-stable} if $\Lambda$ is $\alpha$-stable for $\alpha \in (0, \alpha_1^{})$.
We say that $\Lambda$ is \emph{$\infty$-stable} if $\Lambda$ is $\alpha$-stable for $\alpha \in (\alpha_m^{}, \infty)$.
We write
\begin{alignat*}{3}
& \MXzero && = \MX \qquad && \text{for $\alpha \in (0, \alpha_1^{})$}, \\
& \MXinf && = \MX \qquad && \text{for $\alpha \in (\alpha_m^{}, \infty)$}.
\end{alignat*}
One reason why the spaces $\MXzero$ are interesting to study is the fact that they are closely related to Simpson's moduli spaces.
Indeed, if $(\Gamma, \F)$ is $0^{+}$-stable, then $\F$ is Gieseker-semi-stable.
Thus, we have the \emph{forgetful morphism}
\[
\phi \colon \MXzero \lra \M_X^{}(r, t) \quad \text{given on closed points by} \quad \phi(\langle (\Gamma, \F) \rangle) = \langle \F \rangle.
\]
Notice that
\begin{equation}
\label{phi}
\phi^{-1}(\langle \F \rangle) \simeq \PP(\H^0(\F)) \quad \text{for a stable sheaf} \quad \F.
\end{equation}
By the work of Pandharipande and Thomas \cite{pandharipande_thomas}, the space $\MXinf$ is isomorphic
to a disjoint union of flag Hilbert schemes of points on curves contained in $X$.
At Proposition~\ref{M_infinity} we spell out this isomorphism in the case when $X = \PP^1 \times \PP^1$.
The aim of this paper is to relate the geometry of $\MXzero$ to the geometry of $\MXinf$, which is easier to understand.

Our technique will be to decompose each $\M_X^{\alpha_i}(r, t)^\pss$ into finitely many locally closed subvarieties $S$,
such that the Euler characteristic of $S$, of its right fiber $\fiberR S$ and of its left fiber $\fiberL S$ in diagram~\eqref{wall-crossing}, can be computed.
This technique has already been used by Choi and Chung~\cite{choi_chung}, in the study of moduli spaces of sheaves on $\PP^2$,
and by the author, in the study of moduli spaces of sheaves on $\PP^1 \times \PP^1$ supported on curves of small genus,
see~\cite{dedicata_2}, \cite{algebra}, \cite{advances}.

Our first main result is Theorem~\ref{orbits},
which describes the right and left fibers over a properly $\alpha$-semi-stable point $\langle \Lambda \rangle$.
Let $\Lambda'$ be the term of order $1$ of the Jordan-H\"older filtration of $\Lambda$.
Let $\Theta$ be the direct sum of the terms of order zero, i.e.\ sheaves, of the Jordan-H\"older filtration of $\Lambda$.
Then we have the finite decompositions
\[
\fiberL \langle \Lambda \rangle
= \bigsqcup_{\langle \Lambda'' \rangle = \langle \Theta \rangle} \Ext^1(\Lambda'', \Lambda')^\ses / \Aut(\Lambda'')
\quad \text{and} \quad
\fiberR \langle \Lambda \rangle
= \bigsqcup_{\langle \Lambda'' \rangle = \langle \Theta \rangle} \Ext^1(\Lambda', \Lambda'')^\ses / \Aut(\Lambda'').
\]
The open subsets of \emph{semi-stable extensions} $\Ext^1(\Lambda'', \Lambda')^\ses$ and $\Ext^1(\Lambda', \Lambda'')^\ses$
can be described explicitly.
The quotients modulo $\Aut(\Lambda'')$ are set-theoretic, that is, they are the sets of orbits for the action of $\Aut(\Lambda'')$.
At this time we do not know if these are good or geometric quotients in the sense of GIT.
Our second main result, Theorem~\ref{euler_fibers}, concerns the topological Euler characteristic
$\eulertop(\fiberR \langle \Lambda \rangle)$ and $\eulertop(\fiberL \langle \Lambda \rangle)$
in the particular case when $X = \PP^1 \times \PP^1$ and all terms of the Jordan-H\"older filtration of $\Lambda$
of order zero are supported on lines of degree $(0, 1)$.
In this case $\Aut(\Lambda'') \simeq U \rtimes P$,
where $U$ is a unipotent group and $P$ is a parabolic subgroup of a general linear group.
One approach would be to quotient $\Ext^1(\Lambda', \Lambda'')^\ses$ in stages, first modulo $U$, then modulo $P$.
However, in general it is not known whether quotients modulo unipotent groups, in the sense of GIT, exist.
Instead, we first construct explicitly a geometric quotient $Q = \Ext^1(\Lambda', \Lambda'')^\ses / P$.
We now consider the induced morphism $Q \to \fiberR \langle \Lambda \rangle$.
Roughly speaking, its fibers coincide with the orbits for an algebraic action of $U$ on $Q$.
This allows us to conclude that $\eulertop(\fiberR \langle \Lambda \rangle) = \eulertop(Q)$.

If $X = \PP^1 \times \PP^1$, the moduli space $\M_X^\alpha(\bar{r}, t)$ breaks into disconnected components
$\M_X^\alpha((s, r), t)$ given by the condition $\chern(\F) = (s, r)$.
Here $\bar{r} = r + s$.
Likewise, we define $\M_X^{}((s, r), t)$.
As an application of the main results,
we compute the topological Euler characteristic of $\M_X^{0+}((2, r), t)$ for $t \le 1$ and $3 \le r + t \le 10$,
see equations~\eqref{r+t=3}--\eqref{r+t=10}.
As a consequence, we determine the topological Euler characteristic of $\M_X^{}((2, r), 1)$ for $2 \le r \le 9$,
see equations~\eqref{r=2}--\eqref{r=9}.

The paper is organized as follows.
In section~\ref{right_left} we study the right and left fibers in general.
One important ingredient in this study is the existence of Harder-Narasimhan filtrations of a special kind, see Lemma~\ref{HN}.
In section~\ref{dimension} we restrict our study to the case when $X$ is a surface.
We compute the dimension of  $\Ext^1(\Lambda', \Lambda'')$ and $\Ext^1(\Lambda'', \Lambda')$.
For K3~surfaces and del~Pezzo surfaces we have more specific results, see Proposition~\ref{ext_delPezzo}.
In section~\ref{automorphisms} we show that the group of automorphisms of a skyscraper sheaf $\Z$ on $\PP^1$ is of the form $U \rtimes P$,
where $U$ is a unipotent group and $P$ is a parabolic subgroup of a general linear group.
In section~\ref{topological} we restrict our attention to the case when $X = \PP^1 \times \PP^1$ and $\Lambda'' = \pr_2^* \Z$.
Section~\ref{homomorphisms} deals with certain subsets of the flag Hilbert scheme,
which arise out of the formula for $\eulertop(\fiberL \langle \Lambda \rangle)$.
The results of sections~\ref{topological} and \ref{homomorphisms} are applied in section~\ref{formulas}
to moduli spaces of coherent systems on $\PP^1 \times \PP^1$ with first Chern class $(2, r)$.

\section{Right and left fibers over the properly semi-stable loci}
\label{right_left}

\noindent
In this section $(X, \Osh_X^{}(1))$ will denote a smooth projective polarized variety over $\CC$.
Consider a coherent system $\Lambda = (\Gamma, \F)$ on $X$ of Hilbert polynomial $\Poly_{\F}^{}(x) = r x + t$, where $r > 0$.
The \emph{order} of $\Lambda$ is $\dim_\CC^{} \Gamma$.
The \emph{multiplicity} of $\Lambda$ is $\mult(\Lambda) = r$.
The \emph{Euler characteristic} of $\Lambda$ is $\euler(\Lambda) = t$.
Let $\alpha$ be a positive real number.
The \emph{$\alpha$-slope} of $\Lambda$ with respect to $\Osh_X^{}(1)$ is
\[
\p_{\alpha}^{} (\Lambda) = \frac{\alpha \dim_\CC^{} \Gamma + t}{r}.
\]
We say that $\Lambda$ is \emph{$\alpha$-semi-stable} with respect to $\Osh_X^{}(1)$
if $\F$ is pure, i.e.\ it contains no subsheaves of dimension zero,
and $\p_{\alpha}^{}(\Lambda') \le \p_{\alpha}^{}(\Lambda)$
for any proper coherent subsystem $\Lambda' \subset \Lambda$.
If we impose strict inequality, then $\Lambda$ is said to be \emph{$\alpha$-stable}.
If $\Lambda$ has order zero, i.e.\ $\Lambda$ is a coherent sheaf,
then $\p_{\alpha}^{}(\Lambda) = \p(\Lambda)$ is the usual slope with respect to $\Osh_X^{}(1)$,
and we recover the notion of semi-stability in the sense of Gieseker-Maruyama.

According to \cite[Th\'eor\`eme 4.3]{he}, if $\alpha$ is non-singular, then $\MX$ is a fine moduli space,
i.e.\ there exists a universal family $\LLambda$ of $\alpha$-stable coherent systems on $\MX \times X / \MX$.
Note that $\LLambda$ is flat relative to the base $\MX$.

Consider a singular value $\alpha$ relative to the polynomial $r x + t$.
Let $\epsilon$ be a small positive real number.
Given $\langle \Lambda \rangle \in \M_X^{\alpha + \epsilon}(r, t)$ and a proper subsystem $\Lambda' \subset \Lambda$,
we have the inequality $\p_{\alpha + \epsilon}^{}(\Lambda') < \p_{\alpha + \epsilon}^{}(\Lambda)$.
Taking limit as $\epsilon \to 0$, we obtain the inequality $\p_{\alpha}^{}(\Lambda') \le \p_{\alpha}^{}(\Lambda)$.
This proves that $\Lambda$ is $\alpha$-semi-stable.
The universal family $\LLambda$ on $\M_X^{\alpha + \epsilon}(r, t) \times X / \M_X^{\alpha + \epsilon}(r, t)$ is, therefore,
a flat family of $\alpha$-semi-stable coherent systems.
By the universal property of a coarse moduli space,
$\LLambda$ induces the morphism $\rho_{\alpha + \epsilon}^{}$ from the diagram below.
Analogously, we construct the morphism $\rho_{\alpha - \epsilon}^{}$, so that we obtain the \emph{wall-crossing diagram}
\begin{equation}
\label{wall-crossing}
\xymatrix
{
\M_X^{\alpha - \epsilon}(r, t) \ar[rd]_-{\rho_{\alpha - \epsilon}^{}} & &
\M_X^{\alpha + \epsilon}(r, t) \ar[ld]^-{\rho_{\alpha + \epsilon}^{}} \\
& \MX
}.
\end{equation}
Assume that there exist integers $r' > 0$, $r'' > 0$, $t'$ and $t''$ such that $(\alpha + t') / r' = t'' / r''$, $r = r' + r''$ and $t = t' + t''$.
Assume further that $\MXprime \neq \emptyset$ and $\MXsecond \neq \emptyset$.
Then $\alpha$ is singular and there is a morphism of schemes
\[
\MXprime \times \MXsecond \lra \MX \quad \text{given on closed points by} \quad
(\langle \Lambda' \rangle, \langle \Lambda'' \rangle ) \longmapsto \langle \Lambda' \oplus \Lambda'' \rangle.
\]
The induced morphism of reduced schemes
\[
\MXprime_\red \times \MXsecond_\red^{} \lra \MX^\pss
\]
is injective because the Jordan-H\"older filtration of an $\alpha$-semi-stable coherent system of order $1$
contains precisely one term or order $1$, the others being sheaves.
Thus, the morphism
\[
\gamma \colon \bigsqcup_{r' = 1}^{r - 1}
\M_X^{\alpha} \big( r', \, \tfrac{r'(\alpha + t)}{r} - \alpha \big)^\st_\red \times \ \M_X^{} \big( r - r', \, \tfrac{(r - r')(\alpha + t)}{r} \big)_\red \lra \MX^\pss
\]
induces a bijection between the sets of closed points.
We shall abusively say that $\MX^\pss$ \emph{contains} the space $\MXprime \times \MXsecond$
if the reduction of the latter occurs in the domain of $\gamma$.
Consider a locally closed subscheme $Y \subset \MXprime_\red \times \MXsecond_\red^{}$.
Consider the cartesian diagrams
\[
\xymatrix
{
Y_\Rt^{} \ar[r] \ar[d] & \M_X^{\alpha + \epsilon}(r, t) \ar[d]^-{\rho_{\alpha + \epsilon}}
& \text{and} & Y_\Lt^{} \ar[d] \ar[r] & \M_X^{\alpha - \epsilon}(r, t) \ar[d]^-{\rho_{\alpha - \epsilon}} \\
Y \ar[r]^-{\gamma \circ \incl_Y^{}} & \MX & & Y \ar[r]^-{\gamma \circ \incl_Y^{}} & \MX
}.
\]
We shall write $\fiberR Y = (Y_\Rt^{})_\red^{}$ and $\fiberL Y = (Y_\Lt^{})_\red^{}$.

\begin{proposition}
\label{fibers}
Assume that $\alpha$ is a singular value relative to the polynomial $P(x) = r x + t$.
Assume that $\MX^\pss$ contains the space $\MXprime \times \MXsecond$.
Consider closed points $\langle \Lambda' \rangle \in \MXprime$ and $\langle \Theta \rangle \in \MXsecond$.
\begin{enumerate}
\item[\emph{(i)}]
Assume that $\langle \Lambda \rangle \in \fiberR (\langle \Lambda' \rangle, \langle \Theta \rangle)$.
Then there exists a semi-stable sheaf $\Lambda''$ such that $\langle \Lambda'' \rangle = \langle \Theta \rangle$
and there exists an extension
\[
0 \lra \Lambda'' \lra \Lambda \lra \Lambda' \lra 0
\]
satisfying the following property:
for any stable sheaf $\Delta''$ of the same slope as $\Lambda''$, and for any surjective morphism $\Lambda'' \to \Delta''$,
the image of $\Lambda$ under the induced morphism
\[
\delta_\Rt^{} \colon \Ext^1(\Lambda', \Lambda'') \lra \Ext^1(\Lambda', \Delta'')
\]
is non-zero.
We denote by $\Ext^1(\Lambda', \Lambda'')^\ses$ the set of such extensions.
Conversely, if $\Lambda \in \Ext^1(\Lambda', \Lambda'')^\ses$
for some semi-stable sheaf $\Lambda''$ satisfying $\langle \Lambda'' \rangle = \langle \Theta \rangle$,
then $\Lambda$ is $(\alpha + \epsilon)$-stable and $\langle \Lambda \rangle \in \fiberR(\langle \Lambda' \rangle, \langle \Theta \rangle)$.
\item[\emph{(ii)}]
Assume that $\langle \Lambda \rangle \in \fiberL (\langle \Lambda' \rangle, \langle \Theta \rangle)$.
Then there exists a semi-stable sheaf $\Lambda''$ such that $\langle \Lambda'' \rangle = \langle \Theta \rangle$
and there exists an extension
\[
0 \lra \Lambda' \lra \Lambda \lra \Lambda'' \lra 0
\]
satisfying the following property:
for any stable sheaf $\Delta''$ of the same slope as $\Lambda''$, and for any injective morphism $\Delta'' \to \Lambda''$,
the image of $\Lambda$ under the induced morphism
\[
\delta_\Lt^{} \colon \Ext^1(\Lambda'', \Lambda') \lra \Ext^1(\Delta'', \Lambda')
\]
is non-zero.
We denote by $\Ext^1(\Lambda'', \Lambda')^\ses$ the set of such extensions.
Conversely, if $\Lambda \in \Ext^1(\Lambda'', \Lambda')^\ses$
for some semi-stable sheaf $\Lambda''$ satisfying $\langle \Lambda'' \rangle = \langle \Theta \rangle$,
then $\Lambda$ is $(\alpha - \epsilon)$-stable and $\langle \Lambda \rangle \in \fiberL(\langle \Lambda' \rangle, \langle \Theta \rangle)$.
\end{enumerate}
\end{proposition}

\begin{proof}
Assume that $\langle \Lambda \rangle \in \fiberR(\langle \Lambda' \rangle, \langle \Theta \rangle)$.
The Jordan-H\"older filtration
\[
\{ 0 \} = \Lambda_0^{} \subset \Lambda_1^{} \subset \dots \subset \Lambda_m^{} = \Lambda
\]
of $\Lambda$ with $\alpha$-stable terms has length $m \ge 2$.
For a unique index $i$, $\Lambda_i^{} / \Lambda_{i - 1}^{} \simeq \Lambda'$.
Moreover,
\[
\Bigl< \bigoplus_{1 \le j \le m, \, j \neq i} \Lambda_j^{} / \Lambda_{j - 1}^{} \Bigr> = \langle \Theta \rangle.
\]
If $i \neq m$, then we would obtain the inequality
\[
\p_{\alpha + \epsilon}^{}(\Lambda / \Lambda_{m - 1}^{})
= \p_{\alpha}^{}(\Lambda / \Lambda_{m - 1}^{})
= \p_{\alpha}^{}(\Lambda) < \p_{\alpha + \epsilon}^{}(\Lambda).
\]
This would violate the $(\alpha + \epsilon)$-stability of $\Lambda$.
It follows that $i = m$, so we obtain an extension
\[
0 \lra \Lambda'' \lra \Lambda \lra \Lambda' \lra 0
\]
in which $\Lambda'' = \Lambda_{m - 1}^{}$ is a semi-stable sheaf such that $\langle \Lambda'' \rangle = \langle \Theta \rangle$.
Consider a stable sheaf $\Delta''$ of the same slope as $\Lambda''$ and a surjective morphism $\delta'' \colon \Lambda'' \to \Delta''$.
Write $\Delta = \delta_\Rt^{}(\Lambda)$.
We have the commutative diagram
\[
\xymatrix
{
0 \ar[r] & \Lambda'' \ar[r]^-{\lambda} \ar[d]^-{\delta''} & \Lambda \ar[r] \ar[d]^-{\delta} & \Lambda' \ar[r] \ar@{=}[d] & 0 \\
0 \ar[r] & \Delta'' \ar[r] & \Delta \ar[r] & \Lambda' \ar[r] & 0
}.
\]
Assume that $\Delta$ were a split extension and denote by $\sigma \colon \Delta \to \Delta''$ the splitting morphism.
Notice that $\sigma \circ \delta$ is surjective, because $\sigma \circ \delta \circ \lambda = \delta''$ is surjective.
Thus, $\Delta''$ is a quotient of $\Lambda$ of slope
\[
\p_{\alpha + \epsilon}^{}(\Delta'') = \p(\Delta'') = \p(\Lambda'') = \p_{\alpha}^{}(\Lambda) < \p_{\alpha + \epsilon}^{}(\Lambda).
\]
This violates the $(\alpha + \epsilon)$-stability of $\Lambda$.
We deduce that $\delta_\Rt^{}(\Lambda) \neq \{ 0 \}$.
Conversely, we assume that $\Lambda \in \Ext^1(\Lambda', \Lambda'')^\ses$
for some semi-stable sheaf $\Lambda''$ such that $\langle \Lambda'' \rangle = \langle \Theta \rangle$.
Our goal is to prove that $\Lambda$ is $(\alpha + \epsilon)$-stable.
Consider a proper coherent subsystem $\Sigma = (\Gamma_{\Sigma}^{}, \F_{\Sigma}^{}) \subset \Lambda$.
Let $\Sigma'$ and $\Sigma''$ be the image, respectively, the kernel of the morphism $\Sigma \to \Lambda'$.
We have the commutative diagram
\[
\xymatrix
{
0 \ar[r] & \Sigma'' \ar[r] \ar[d] & \Sigma \ar[r] \ar[d] & \Sigma' \ar[r] \ar[d] & 0 \\
0 \ar[r] & \Lambda'' \ar[r] & \Lambda \ar[r] & \Lambda' \ar[r] & 0
}.
\]
Assume, first, that $\Sigma' = \{ 0 \}$.
Using the semi-stability of $\Lambda''$, we obtain the relations
\[
\p_{\alpha + \epsilon}^{}(\Sigma) = \p(\Sigma'') \le \p(\Lambda'') = \p_{\alpha}^{}(\Lambda) < \p_{\alpha + \epsilon}^{}(\Lambda).
\]
Assume, secondly, that $\Sigma' \neq \{ 0 \}$ and $\Sigma' \neq \Lambda'$.
Since $\Lambda'$ is $\alpha$-stable, we have the inequality
\[
\p_{\alpha}^{}(\Sigma') < \p_{\alpha}^{}(\Lambda') = \p_{\alpha}^{}(\Lambda), \quad \text{that is}, \quad
\alpha \dim_\CC^{} \Gamma_{\Sigma}^{} + \euler(\Sigma') < \mult(\Sigma') \p_{\alpha}^{}(\Lambda).
\]
We can find $\epsilon \in (0, \infty)$ sufficiently small and independent of $\Sigma$ such that
\[
(\alpha + \epsilon) \dim_\CC^{} \Gamma_{\Sigma}^{} + \euler(\Sigma') < \mult(\Sigma') \p_{\alpha}^{}(\Lambda).
\]
From the semi-stability of $\Lambda''$ we obtain the inequality $\euler(\Sigma'') \le \mult(\Sigma'') \p_{\alpha}^{}(\Lambda)$.
Adding the last two inequalities yields the inequalities
\[
(\alpha + \epsilon) \dim_\CC^{} \Gamma_{\Sigma}^{} + \euler(\Sigma)
< \mult(\Sigma) \p_{\alpha}^{}(\Lambda) < \mult(\Sigma) \p_{\alpha + \epsilon}^{}(\Lambda).
\]
Thus, $\p_{\alpha + \epsilon}^{}(\Sigma) < \p_{\alpha + \epsilon}^{}(\Lambda)$.
It remains to examine the case when $\Sigma' = \Lambda'$.
Notice that $\Sigma'' \neq \{ 0 \}$ because, by hypothesis,
$\Lambda$ is a non-split extension of $\Lambda'$ by $\Lambda''$.
We claim that $\p(\Sigma'') < \p(\Lambda'')$.
Assume, on the contrary, that $\p(\Sigma'') = \p(\Lambda'')$.
From the snake lemma applied to the above diagram we obtain the commutative diagram
\[
\xymatrix
{
0 \ar[r] & \Lambda'' \ar[r] \ar@{->>}[d] & \Lambda \ar[r] \ar@{->>}[d] & \Lambda' \ar[r] & 0 \\
& \Lambda'' / \Sigma'' \ar@{=}[r] & \Lambda'' / \Sigma''
}
\]
in which $\Lambda'' / \Sigma''$ is a semi-stable sheaf of slope $\p(\Lambda'')$.
Choose a stable quotient sheaf $\Delta''$ of $\Lambda'' / \Sigma''$ of slope $\p(\Delta'') = \p(\Lambda'')$.
From the commutative diagram
\[
\xymatrix
{
0 \ar[r] & \Lambda'' \ar[r]^-{\lambda} \ar@{->>}[d]^-{\delta''} & \Lambda \ar[r] \ar@{->>}[d]^-{\delta} & \Lambda' \ar[r] & 0 \\
& \Delta'' \ar@{=}[r] & \Delta''
}
\]
we obtain the relations
\[
\delta_\Rt^{}(\Lambda) = \Ext^1(\Lambda', \delta'')(\Lambda) = \Ext^1(\Lambda', \delta)(\Ext^1(\Lambda', \lambda)(\Lambda))
= \Ext^1(\Lambda', \delta)(\{ 0 \}) = \{ 0 \}.
\]
This contradicts our choice of $\Lambda$ in $\Ext^1(\Lambda', \Lambda'')^\ses$.
We have proved the claim.
From the inequality $\p(\Sigma'') < \p_{\alpha}^{}(\Lambda)$ we obtain the inequality
$\euler(\Sigma'') < \mult(\Sigma'') \p_{\alpha}^{}(\Lambda)$.
We can find $\epsilon \in (0, \infty)$ sufficiently small and independent of $\Sigma$
such that $\epsilon + \euler(\Sigma'') < \mult(\Sigma'') \p_{\alpha}^{}(\Lambda)$.
This inequality and the equation $\alpha + \euler(\Sigma') = \mult(\Sigma') \p_{\alpha}^{}(\Lambda)$ lead to the inequalities
\[
(\alpha + \epsilon) \dim_\CC^{} \Gamma_{\Sigma}^{} + \euler(\Sigma) < \mult(\Sigma) \p_{\alpha}^{}(\Lambda)
< \mult(\Sigma) \p_{\alpha + \epsilon}^{}(\Lambda).
\]
Thus, $\p_{\alpha + \epsilon}^{}(\Sigma) < \p_{\alpha + \epsilon}^{}(\Lambda)$.
We conclude that $\Lambda$ is $(\alpha + \epsilon)$-stable, which proves (i).
Part (ii) can be proved analogously, by dualising the above arguments.
\end{proof}

\begin{lemma}
\label{HN}
We assume that $\Lambda = (\Gamma, \F)$ gives a stable point in $\MX$.
We make the following claims:
\begin{enumerate}
\item[\emph{(i)}]
There exists a filtration $\{ 0 \} = \Lambda_0^{} \subset \dots \subset \Lambda_k^{} = \Lambda$ of $\Lambda$ by coherent subsystems,
with $\Lambda_i^{} = (\{ 0 \}, \F_i^{})$ for $0 \le i \le k - 1$,
such that $\F_i^{} / \F_{i - 1}^{}$ is semi-stable for $1 \le i \le k - 1$, $(\Gamma, \F / \F_{k - 1}^{})$ is $0^+$-stable, and
\[
\p_{\alpha}^{}(\Lambda) > \p(\F_1^{} / \F_0^{}) > \p(\F_2^{} / \F_1^{}) > \dots > \p(\F / \F_{k - 1}^{}).
\]
This is the Harder-Narasimhan filtration of $\Lambda$ with $0^+$-semi-stable terms.
\item[\emph{(ii)}]
There exists a filtration $\{ 0 \} = \Delta_0^{} \subset \dots \subset \Delta_l^{} = \Lambda$ of $\Lambda$ by coherent subsystems,
with $\Delta_i^{} = (\Gamma, \G_i^{})$ for $1 \le i \le l$,
such that $\Delta_1^{}$ is $\infty$-stable, $\G_i^{} / \G_{i - 1}^{}$ is semi-stable for $2 \le i \le l$, and
\[
\p(\G_2^{} / \G_1^{}) > \dots > \p(\G_l^{} / \G_{l - 1}^{}) > \p_{\alpha}^{}(\Lambda).
\]
This is the Harder-Narasimhan filtration of $\Lambda$ with $\infty$-semi-stable terms.
\end{enumerate}
\end{lemma}

\begin{proof}
To prove (i) we perform induction on $\mult(\Lambda)$.
Assume that $\mult(\Lambda) = 1$.
There are no singular values relative to the polynomial $x + t$, hence $\Lambda$ is $0^+$-stable,
and hence the trivial filtration of length $k = 1$ has the required properties.
Assume that $\mult(\Lambda) > 1$.
We may also assume that $\Lambda$ is not $0^+$-stable, otherwise we take the trivial filtration.
There exists a singular value $\beta$ relative to the polynomial $r x + t$,
such that $\beta < \alpha$ and such that $\Lambda$ is properly $\beta$-semi-stable.
We take $\beta$ to be maximal possible, which ensures that $\Lambda$ be $(\beta + \epsilon)$-stable.
According to Proposition~\ref{fibers}(i), we have an extension
\[
0 \lra \Lambda'' \lra \Lambda \lra \Lambda' \lra 0,
\]
such that $\Lambda''$ is a semi-stable sheaf, $\Lambda'$ is $\beta$-stable
and $\p_\beta^{}(\Lambda) = \p(\Lambda'') = \p_\beta^{}(\Lambda')$.
By the induction hypothesis, there exists a filtration
\[
\{ 0 \} = \Lambda_1' \subset \dots \subset \Lambda_k' = \Lambda' = (\Gamma', \F'),
\]
with $\Lambda_i' = (\{ 0 \}, \F_i')$ for $1 \le i \le k - 1$,
such that $\F_i' / \F_{i - 1}'$ is semi-stable for $2 \le i \le k - 1$, $(\Gamma', \F' / \F_{k - 1}')$ is $0^+$-stable, and
\[
\p_\beta^{}(\Lambda') > \p(\F_2' / \F_1') > \dots > \p(\F' / \F_{k - 1}').
\]
For $1 \le i \le k - 1$, let $\Lambda_i^{} = (\{ 0 \}, \F_i^{})$ be the preimage of $\Lambda_i'$ in $\Lambda$.
Put $\Lambda_0^{} = \{ 0 \}$, $\Lambda_k^{} = \Lambda$.
The filtration $\{ \Lambda_i^{} \}_{0 \le i \le k}^{}$ satisfies the requirements of the lemma because
\[
\p_{\alpha}^{}(\Lambda) > \p_\beta^{}(\Lambda)
= \p(\Lambda'') = \p(\F_1^{})
= \p_\beta^{}(\Lambda')
> \p(\F_2^{} / \F_1^{}) > \dots > \p(\F / \F_{k - 1}^{}).
\]
This concludes the proof of (i).
Part (ii) can be proved analogously, by dualising the above arguments.
If $\Lambda$ is not $\infty$-stable,
then we take $\beta > \alpha$ to be the smallest singular value such that $\Lambda$ is properly $\beta$-semi-stable
and we invoke Proposition~\ref{fibers}(ii).
\end{proof}

\noindent
The filtration $\{ 0 \} \subset \F_1^{} \subset \dots \subset \F_{k - 1}^{} \subset \F$ from Lemma~\ref{HN}(i)
is the usual Harder-Narasimham filtration of $\F$.

Assume that $\alpha$ is a singular value relative to the polynomial $r x + t$.
Assume that $\MX^\pss$ contains the space $\MXprime \times \MXsecond$.
Consider closed points $\langle \Lambda' \rangle \in \MXprime$ and $\langle \Lambda'' \rangle \in \MXsecond$.
Denote $E = \Spec(\Sym(\Ext^1(\Lambda', \Lambda'')^*))$.
Applying \cite[Corollary 4.3.3]{tommasini} or \cite[Corollary 3.4]{lange}
to the morphism $X \to \{ \pt \}$ we obtain a universal extension
\[
0 \lra \Lambda_E'' \lra \LLambda \lra \Lambda_E' \lra 0,
\]
of families of coherent systems on $E \times X/ E$,
whose restriction to every closed point $\Lambda \in \Ext^1(\Lambda', \Lambda'')$
lies in the same extension class as $\Lambda$.
The subset
\[
\Ext^1(\Lambda', \Lambda'')^\ses =
\{ \Lambda \in \Ext^1(\Lambda', \Lambda'') \mid \LLambda |_{\{ \Lambda \} \times X} \ \text{is $(\alpha + \epsilon)$-stable} \}
\]
from Proposition~\ref{fibers}(i) is open because $\LLambda$ is flat over $E$.
We denote by $E^\ses \subset E$ the corresponding open subscheme.
Thus, $\LLambda |_{E^\ses}$ is a flat family of $(\alpha + \epsilon)$-stable coherent systems on $X$
that have order $1$ and Hilbert polynomial $r x + t$.
By the universal property of $\M_X^{\alpha + \epsilon}(r, t)$,
$\LLambda |_{E^\ses}$ gives rise to a morphism of schemes $E^\ses \to \M_X^{\alpha + \epsilon}(r, t)$.
The induced morphism of reduced schemes has image contained in $\fiberR(\langle \Lambda' \rangle, \langle \Lambda'' \rangle)$,
so it restricts to a morphism of varieties
\[
\eta_\Rt^{} \colon \Ext^1(\Lambda', \Lambda'')^\ses \lra \fiberR(\langle \Lambda' \rangle, \langle \Lambda'' \rangle)
\quad \text{given by} \quad \eta_\Rt^{}(\Lambda) = \langle \Lambda \rangle.
\]
Analogously, employing the universal family of extensions of $\Lambda''$ by $\Lambda'$ and Proposition~\ref{fibers}(ii),
we deduce that the subset $\Ext^1(\Lambda'', \Lambda')^\ses \subset \Ext^1(\Lambda'', \Lambda')$
of $(\alpha - \epsilon)$-stable extensions is open
and we obtain the morphism of varieties
\[
\eta_\Lt^{} \colon \Ext^1(\Lambda'', \Lambda')^\ses \lra \fiberL(\langle \Lambda' \rangle, \langle \Lambda'' \rangle)
\quad \text{given by} \quad \eta_\Lt^{}(\Lambda) = \langle \Lambda \rangle.
\]
The group $\Aut(\Lambda') \times \Aut(\Lambda'')$ acts by conjugation on $\Ext^1(\Lambda', \Lambda'')$ and on $\Ext^1(\Lambda'', \Lambda')$.
The open subsets of $(\alpha + \epsilon)$-stable extensions, respectively, of $(\alpha - \epsilon)$-stable extensions are invariant.
The morphisms $\eta_\Rt^{}$ and $\eta_\Lt^{}$ are $\Aut(\Lambda') \times \Aut(\Lambda'')$-equivariant.
Since $\Lambda'$ is $\alpha$-stable, it follows that $\Aut(\Lambda') \simeq \CC^*$.
The subgroup of homotheties $\{ (a \id_{\Lambda'}^{}, a \id_{\Lambda''}^{}) \mid a \in \CC^* \}$ acts trivially,
hence the action on either space factors through an action of
\[
(\Aut(\Lambda') \times \Aut(\Lambda'')) / \{ (a \id_{\Lambda'}^{}, a \id_{\Lambda''}^{}) \mid a \in \CC^* \}
\simeq(\CC^* \times \Aut(\Lambda'')) / \CC^*
\simeq \Aut(\Lambda'').
\]
In the next proposition we fix $\langle \Theta \rangle \in \MXsecond$ and we denote by $\If(\Theta)$
the set of isomorphism classes of semi-stable sheaves $\Lambda''$ satisfying $\langle \Lambda'' \rangle = \langle \Theta \rangle$.
It is known that $\If(\Theta)$ is finite.

\begin{theorem}
\label{orbits}
We adopt the above assumptions and notations.
We make the following claims:
\begin{enumerate}
\item[\emph{(i)}]
The fibers of $\eta_\Rt^{}$ are precisely the $\Aut(\Lambda'')$-orbits.
We have the finite decomposition
\[
\fiberR (\langle \Lambda' \rangle, \langle \Theta \rangle)
= \bigsqcup_{\Lambda'' \in \If(\Theta)} \eta_\Rt^{} (\Ext^1(\Lambda', \Lambda'')^\ses).
\]
\item[\emph{(ii)}]
The fibers of $\eta_\Lt^{}$ are precisely the $\Aut(\Lambda'')$-orbits.
We have the finite decomposition
\[
\fiberL (\langle \Lambda' \rangle, \langle \Theta \rangle)
= \bigsqcup_{\Lambda'' \in \If(\Theta)} \eta_\Lt^{} (\Ext^1(\Lambda'', \Lambda')^\ses).
\]
\end{enumerate}
\end{theorem}

\begin{proof}
Assume that, for $\Lambda$ and $\Delta$ in $\Ext^1(\Lambda', \Lambda'')^\ses$,
we have $\eta_\Rt^{}(\Lambda) = \eta_\Rt^{}(\Delta)$.
Two $(\alpha + \epsilon)$-stable coherent systems are S-equivalent if and only if they are isomorphic,
hence there is an isomorphism $\lambda \colon \Lambda \to \Delta$.
We seek $\lambda' \in \Aut(\Lambda')$ and $\lambda'' \in \Aut(\Lambda'')$ making the diagram
\[
\xymatrix
{
0 \ar[r] & \Lambda'' \ar[r] \ar[d]^-{\lambda''} & \Lambda \ar[r] \ar[d]^-{\lambda} & \Lambda' \ar[r] \ar[d]^-{\lambda'} & 0 \\
0 \ar[r] & \Lambda'' \ar[r] & \Delta \ar[r] & \Lambda' \ar[r] & 0
}
\]
commute.
Recall, from Lemma~\ref{HN}(i), the Harder-Narasimhan filtration $\{ \Lambda_i^{} \}_{0 \le i \le k}^{}$ of $\Lambda$
with $0^+$-semi-stable terms.
The proof of Lemma~\ref{HN}(i) shows that $\Lambda_1^{} = \Lambda''$ and that $\{ \Lambda_i^{} \}_{2 \le i \le k}^{}$
is the preimage in $\Lambda$ of the Harder-Narasimhan filtration of $\Lambda'$ with $0^+$-semi-stable terms.
The Harder-Narasimhan filtration with $0^+$-semi-stable terms of $\Delta$ can be obtained in a similar fashion.
Any isomorphism $\Lambda \to \Delta$ preserves the Harder-Narasimhan filtrations,
so $\lambda$ induces $\lambda''$ making the first square of the above diagram commute.
We take $\lambda'$ to be the induced isomorphism.
We deduce that $\Lambda$ and $\Delta$ lie in the same $\Aut(\Lambda') \times \Aut(\Lambda'')$-orbit.
By virtue of Proposition~\ref{fibers}(i),
\[
\fiberR (\langle \Lambda' \rangle, \langle \Theta \rangle)
= \bigcup_{\Lambda'' \in \If(\Theta)} \eta_\Rt^{} (\Ext^1(\Lambda', \Lambda'')^\ses).
\]
It remains to prove that the above union is disjoint.
Assume that $\eta_\Rt^{} (\Ext^1(\Lambda', \Lambda'')^\ses)$ and $\eta_\Rt^{} (\Ext^1(\Lambda', \Delta'')^\ses)$ have a common point.
This point is represented by isomorphic coherent systems $\Lambda \in \Ext^1(\Lambda', \Lambda'')^\ses$
and $\Delta \in \Ext^1(\Lambda', \Delta'')^\ses$.
Arguing as above, we can show that an isomorphism $\lambda \colon \Lambda \to \Delta$
induces isomorphisms $\lambda'$ and $\lambda''$ such that the diagram
\[
\xymatrix
{
0 \ar[r] & \Lambda'' \ar[r] \ar[d]^-{\lambda''} & \Lambda \ar[r] \ar[d]^-{\lambda} & \Lambda' \ar[r] \ar[d]^-{\lambda'} & 0 \\
0 \ar[r] & \Delta'' \ar[r] & \Delta \ar[r] & \Lambda' \ar[r] & 0
}
\]
becomes commutative.
In particular, $\Lambda''$ and $\Delta''$ lie in the same isomorphism class.
This concludes the proof of (i).
To prove part (ii) we exploit Lemma~\ref{HN}(ii) and of Proposition~\ref{fibers}(ii).
\end{proof}

\begin{proposition}
\label{fibers_stable}
We assume that $\alpha$ is a singular value relative to the polynomial $P(x) = r x + t$.
We assume that $\MX^\pss$ contains the space $\MXprime \times \MXsecond$.
We consider closed and stable points $\langle \Lambda' \rangle \in \MXprime$ and $\langle \Lambda'' \rangle \in \MXsecond^\st$.
\begin{enumerate}
\item[\emph{(i)}]
We claim that there is a bijective morphism of varieties
\[
\theta_\Rt^{} \colon \PP(\Ext^1(\Lambda', \Lambda'')) \lra \fiberR(\langle \Lambda' \rangle, \langle \Lambda'' \rangle)
\quad \text{given by} \quad
\theta_\Rt^{}(\CC \Lambda) = \langle \Lambda \rangle.
\]
\item[\emph{(ii)}]
We claim that there is a bijective morphism of varieties
\[
\theta_\Lt^{} \colon \PP(\Ext^1(\Lambda'', \Lambda')) \lra \fiberL(\langle \Lambda' \rangle, \langle \Lambda'' \rangle)
\quad \text{given by} \quad
\theta_\Lt^{}(\CC \Lambda) = \langle \Lambda \rangle.
\]
\end{enumerate}
\end{proposition}

\begin{proof}
Any sheaf that is S-equivalent to $\Lambda''$ must be isomorphic to $\Lambda''$ because the latter is assumed to be stable.
From Theorem~\ref{orbits}(i) we have
\[
\fiberR(\langle \Lambda' \rangle, \langle \Lambda'' \rangle) = \eta_\Rt^{}(\Ext^1(\Lambda', \Lambda'')^\ses).
\]
The surjective morphism $\Lambda'' \to \Delta''$ in Proposition~\ref{fibers}(i) must be an isomorphism,
hence $\delta_\Rt^{}$ is an isomorphism, and hence
\[
\Ext^1(\Lambda', \Lambda'')^\ses = \Ext^1(\Lambda', \Lambda'') \smallsetminus \{ 0 \}.
\]
According to Theorem~\ref{orbits}(i), the fibers of $\eta_\Rt^{}$ are precisely the orbits modulo the action of
$\Aut(\Lambda'') \simeq \CC^*$ by multiplication.
It follows that $\eta_\Rt^{}$ factors through a bijective morphism $\theta_\Rt^{}$.
This concludes the proof of (i).
To prove part (ii) we take advantage of Proposition~\ref{fibers}(ii) and of Theorem~\ref{orbits}(ii).
\end{proof}

\noindent
It can be proved that $\theta_\Rt^{}$ and $\theta_\Lt^{}$ are isomorphisms, but we do not need this fact.
In \cite{choi_chung} one can find the particular cases of Propositions~\ref{fibers} and \ref{fibers_stable}
in which $X = \PP^2$, $\Theta$ is stable, and $P(x)$ is one of the following: $4 x + 1$, $4 x + 3$, $5 x + 1$.

\section{The dimension of the extension spaces}
\label{dimension}

\noindent
In this section we shall restrict our study to the case when $X$ is a smooth projective polarized surface over $\CC$.
We shall compute the dimension of the extension spaces $\Ext^1(\Lambda', \Lambda'')$ and $\Ext^1(\Lambda'', \Lambda')$
occurring in Theorem~\ref{orbits}.
Given coherent systems $\Lambda_1^{}$ and $\Lambda_2^{}$ on $X$ we use the standard notation
\[
\euler(\Lambda_1^{}, \Lambda_2^{}) = \sum_{i \ge 0} (-1)^i \ext^i(\Lambda_1^{}, \Lambda_2^{}).
\]
For the convenience of the reader we include the following well-known lemma.

\begin{lemma}
\label{chi_sheaves}
Let $X$ be a smooth projective surface over $\CC$.
Consider coherent $\Osh_X^{}$-modules $\F$ and $\G$.
Then we have the equation
\[
\euler(\F, \G) = - \rank(\F) \rank(\G) \euler(\Osh_X^{}) + \rank(\G) \euler(\F \tensor \omega_X^{})
+ \rank(\F) \euler(\G) - \langle \chern(\F), \chern(\G) \rangle.
\]
\end{lemma}

\begin{proof}
Both sides in the above equation can be viewed as biadditive maps $\Grot(X) \times \Grot(X) \to \ZZ$,
where $\Grot(X)$ denotes the Grothendieck group of $X$.
The classes of line bundles generate $\Grot(X)$,
hence it is enough to prove the lemma in the particular case when $\F$ and $\G$ are line bundles.
Using the Hirzebruch-Riemann-Roch theorem, we calculate:
\begin{align*}
\euler(\F, \G) = {} & \euler(\F^{-1} \tensor \G) \\
= {} & \euler(\Osh_X^{}) +
\frac{1}{2} \langle \chern(\F^{-1} \tensor \G), \, \chern(\F^{-1} \tensor \G) \rangle - \frac{1}{2} \langle \chern(\F^{-1} \tensor \G), \, \chern(\omega_X^{}) \rangle \\
= {} & \euler(\Osh_X^{}) + \frac{1}{2} \langle \chern(\F^{-1}), \chern(\F^{-1}) \rangle + \frac{1}{2} \langle \chern(\G), \chern(\G) \rangle
+ \langle \chern(\F^{-1}), \chern(\G) \rangle \\
& - \frac{1}{2} \langle \chern(\F^{-1}), \chern(\omega_X^{}) \rangle - \frac{1}{2} \langle \chern(\G), \chern(\omega_X^{}) \rangle \\
= {} & - \euler(\Osh_X^{})
+ \Big( \euler(\Osh_X^{}) + \frac{1}{2} \langle \chern(\F^{-1}), \chern(\F^{-1}) \rangle - \frac{1}{2} \langle \chern(\F^{-1}), \chern(\omega_X^{}) \rangle \Big) \\
& + \Big( \euler(\Osh_X^{}) + \frac{1}{2} \langle \chern(\G), \chern(\G) \rangle - \frac{1}{2} \langle \chern(\G), \chern(\omega_X^{}) \rangle \Big)
- \langle \chern(\F), \chern(\G) \rangle \\
= {} & - \euler(\Osh_X^{}) + \euler(\F^{-1}) + \euler(\G) - \langle \chern(\F), \chern(\G) \rangle \\
= {} & - \euler(\Osh_X^{}) + \euler(\F \tensor \omega_X^{}) + \euler(\G) - \langle \chern(\F), \chern(\G) \rangle. \qedhere
\end{align*}
\end{proof}

\begin{proposition}
\label{hom_vanishes}
We assume that $\alpha$ is a singular value relative to the polynomial $P(x) = r x + t$.
We assume that $\MX^\pss$ \emph{contains} the space $\MXprime \times \MXsecond$.
We consider closed points $\langle \Lambda' \rangle \in \MXprime$ and $\langle \Lambda'' \rangle \in \MXsecond$.
We make the following claims:
\begin{align*}
\tag{i} \Hom(\Lambda', \Lambda'') = \{ 0 \}; \\
\tag{ii} \Hom(\Lambda'', \Lambda') = \{ 0 \}.
\end{align*}
\end{proposition}

\begin{proof}
Write $\Lambda' = (\Gamma', \F')$ and $\Lambda'' = (\Gamma'', \F'') = (\{ 0 \}, \F'')$. 
Any non-zero morphism $\Lambda' \to \Lambda''$ would have to be injective,
because $\Lambda'$ is $\alpha$-stable, $\Lambda''$ is $\alpha$-semi-stable and $\p_{\alpha}^{}(\Lambda') = \p_{\alpha}^{}(\Lambda'')$.
However, there is no injective morphism $\Gamma' \to \Gamma''$.
Likewise, any non-zero morphism $\Lambda'' \to \Lambda'$ would have to be surjective,
yet there is no surjective morphism $\Gamma'' \to \Gamma'$.
\end{proof}

\begin{proposition}
\label{ext_surface}
Let $X$ be a smooth projective polarized surface over $\CC$.
Assume that $\alpha$ is a singular value relative to the polynomial $P(x) = r x + t$.
Assume that $\MX^\pss$ \emph{contains} the space $\MXprime \times \MXsecond$.
Consider closed points $\langle \Lambda' \rangle \in \MXprime$ and $\langle \Lambda'' \rangle \in \MXsecond$.
Write $\Lambda' = (\Gamma', \F')$ and $\Lambda'' = (\Gamma'', \F'') = (\{ 0 \}, \F'')$.
Then we have the following formulas:
\begin{align*}
\tag{i}
\ext^1(\Lambda', \Lambda'') & = \langle \chern(\Lambda'), \chern(\Lambda'') \rangle + \euler(\Lambda'') + \ext^2(\Lambda', \Lambda''); \\
\tag{ii}
\ext^1(\Lambda'', \Lambda') & = \langle \chern(\Lambda'), \chern(\Lambda'') \rangle + \hom(\F', \F'' \tensor \omega_X^{}).
\end{align*}
\end{proposition}

\begin{proof}
(i) According to \cite[Corollaire 1.6]{he}, we have the exact sequence
\begin{align*}
0 & \lra \Hom(\Lambda', \Lambda'') \lra \Hom(\F', \F'') \lra \Hom(\Gamma', \H^0(\F'') / \Gamma'') \\
& \lra \Ext^1(\Lambda', \Lambda'') \lra \Ext^1(\F', \F'') \lra \Hom(\Gamma', \H^1(\F'')) \\
& \lra \Ext^2(\Lambda', \Lambda'') \lra \Ext^2(\F', \F'') \lra \Hom(\Gamma', \H^2(\F'')) = \{ 0 \}.
\end{align*}
This yields the relation $\euler(\Lambda', \Lambda'') = \euler(\F', \F'') - \euler(\F'')$.
From Lemma~\ref{chi_sheaves}, and taking into account that both $\F'$ and $\F''$ have rank zero, we obtain the formula
\[
\euler(\F', \F'') = - \langle \chern(\F'), \chern(\F'') \rangle.
\]
Thus,
\[
\hom(\Lambda', \Lambda'') - \ext^1(\Lambda', \Lambda'') + \ext^2(\Lambda', \Lambda'')
= - \langle \chern(\Lambda'), \chern(\Lambda'') \rangle - \euler(\Lambda'').
\]
Formula (i) follows from the vanishing of $\Hom(\Lambda', \Lambda'')$, proved at Proposition~\ref{hom_vanishes}(i).

\medskip

\noindent
(ii) According to \cite[Corollaire 1.6]{he}, we have the exact sequence
\begin{align*}
0 & \lra \Hom(\Lambda'', \Lambda') \lra \Hom(\F'', \F') \lra \Hom(\Gamma'', \H^0(\F') / \Gamma') = \{ 0 \} \\
& \lra \Ext^1(\Lambda'', \Lambda') \lra \Ext^1(\F'', \F') \lra \Hom(\Gamma'', \H^1(\F')) = \{ 0 \} \\
& \lra \Ext^2(\Lambda'', \Lambda') \lra \Ext^2(\F'', \F') \lra \Hom(\Gamma'', \H^2(\F')) = \{ 0 \}.
\end{align*}
We obtain the relations
\begin{align*}
\ext^1(\Lambda'', \Lambda') = \ext^1(\F'', \F') & = - \euler(\F'', \F') + \hom(\F'', \F') + \ext^2(\F'', \F') \\
& = \langle \chern(\Lambda'), \chern(\Lambda'') \rangle + \hom(\Lambda'', \Lambda') + \ext^2(\F'', \F').
\end{align*}
Formula (ii) follows from the vanishing of $\Hom(\Lambda'', \Lambda')$, proved at Proposition~\ref{hom_vanishes}(ii), and from Serre duality.
\end{proof}

\begin{proposition}
\label{ext_delPezzo}
We adopt the assumptions of Proposition~\ref{ext_surface}.
We assume that $X$ is a del~Pezzo surface or a K3 surface.
\begin{enumerate}
\item[\emph{(i)}]
We assume, in addition, that $\H^1(\F'') = \{ 0 \}$.
Then $\Ext^2(\Lambda', \Lambda'') = \{ 0 \}$ and
\[
\ext^1(\Lambda', \Lambda'') = \langle \chern(\Lambda'), \chern(\Lambda'') \rangle + \euler(\Lambda'').
\]
\item[\emph{(ii)}]
We assume, in addition, that $\H^0(\F'' \tensor \omega_X^{}) = \{ 0 \}$.
Then $\Hom(\F', \F'' \tensor \omega_X^{}) = \{ 0 \}$ and
\[
\ext^1(\Lambda'', \Lambda') = \langle \chern(\Lambda'), \chern(\Lambda'') \rangle. \phantom{{} + \euler(\Lambda'')}
\]
\end{enumerate}
\end{proposition}

\begin{proof}
If $X$ is a del~Pezzo surface and $\F$ is a pure sheaf of dimension $1$ on $X$,
then there is a global section $s \in \H^0(-\omega_X^{})$ that is a non-zero-divisor relative to $\F$,
i.e.\ the map $\F \tensor \omega_X^{} \xrightarrow{\cdot s} \F$ is injective.
This is equivalent to saying that the zero-set of $s$ does not contain any irreducible component of the support of $\F$.
This statement follows from the fact that $\left| -\omega_X^{} \right|$ has no fixed part, see \cite[Theorem 8.3.2]{dolgachev}

\medskip

\noindent
(i) According to \cite[Corollaire 1.6]{he}, we have the exact sequence
\[
\{ 0 \} = \Hom(\Gamma', \H^1(\F'')) \lra \Ext^2(\Lambda', \Lambda'') \lra \Ext^2(\F', \F'').
\]
We reduce the problem to showing that $\Ext^2(\F', \F'') = \{ 0 \}$.
By Serre duality, this is equivalent to the vanishing of $\Hom(\F'', \F' \tensor \omega_X^{})$.
Under either hypothesis that $X$ be a del~Pezzo surface or a K3 surface,
$\Hom(\F'', \F' \tensor \omega_X^{})$ is isomorphic to a subspace of $\Hom(\F'', \F')$.
According to Proposition~\ref{hom_vanishes}(ii), $\Hom(\F'', \F') \simeq \Hom(\Lambda'', \Lambda')$ vanishes.
The desired expression for $\ext^1(\Lambda', \Lambda'')$ follows from Proposition~\ref{ext_surface}(i).

\medskip

\noindent
(ii) As per Lemma~\ref{HN}(ii), we consider the Harder-Narasimhan filtration of $\Lambda'$ with $\infty$-semi-stable terms
\[
\{ 0 \} = \Delta_0^{} \subset \dots \subset \Delta_l^{} = \Lambda',
\]
where $\Delta_i^{} = (\Gamma', \G_i^{})$ for $1 \le i \le l$.
We claim that $\Hom(\G_i^{} / \G_{i - 1}^{}, \, \F'' \tensor \omega_X^{})$ vanishes for $2 \le i \le l$.
Under either hypothesis that $X$ be a del~Pezzo surface or a K3 surface,
this space is isomorphic to a subspace of $\Hom(\G_i^{} / \G_{i - 1}^{}, \, \F'')$.
For $2 \le i \le l$, we have the relations
\[
\p(\F'') = \p_{\alpha}^{}(\Lambda'') = \p_{\alpha}^{}(\Lambda') < \p(\G_i^{} / \G_{i - 1}).
\]
Since both $\F''$ and $\G_i^{} / \G_{i - 1}^{}$ are semi-stable, we deduce that $\Hom(\G_i^{} / \G_{i - 1}, \F'') = \{ 0 \}$.
This proves the claim.
Let $\Osh_C^{}$ be the subsheaf of $\G_1^{}$ generated by $\Gamma'$.
Write $\Poly_{\G_1}^{}(x) = r_1^{} x + t_1^{}$, $\Poly_{\Osh_C}^{}(x) = r_1' x + t_1'$.
If $r_1' < r_1^{}$, then
\[
\p_\beta^{}(\Gamma', \Osh_C^{})
= \frac{\beta + t_1'}{r_1'} > \frac{\beta + t_1^{}}{r_1^{}}
= \p_\beta^{}(\Delta_1^{}) \quad \text{for} \quad \beta \gg 0.
\]
This would contradict the fact that $\Delta_1^{}$ is $\infty$-semi-stable.
Thus, $r_1' = r_1^{}$, $\Poly_{\G_1 / \Osh_C}^{}(x) = t_1^{} - t_1'$,
hence $\G_1^{} / \Osh_C^{}$ is supported on finitely many points or is zero.
From the exact sequences
\[
\{ 0 \} = \Hom(\G_i^{} / \G_{i - 1}^{}, \F'' \tensor \omega_X^{}) \lra \Hom(\G_i^{}, \F'' \tensor \omega_X^{})
\lra \Hom(\G_{i - 1}^{}, \F'' \tensor \omega_X^{})
\]
for $2 \le i \le l$ and
\[
\{ 0 \} = \Hom(\G_1^{} / \Osh_C^{}, \F'' \tensor \omega_X^{}) \lra \Hom(\G_1^{}, \F'' \tensor \omega_X^{})
\lra \Hom(\Osh_C^{}, \F'' \tensor \omega_X^{})
\]
we deduce that $\Hom(\F', \F'' \tensor \omega_X^{})$ can be embedded in $\Hom(\Osh_C^{}, \F'' \tensor \omega_X^{})$.
This space can be embedded in $\H^0(\F'' \tensor \omega_X^{})$ because $\Osh_C^{}$ is generated by a global section.
By hypothesis, the latter space vanishes.
The desired expression for $\ext^1(\Lambda'', \Lambda')$ follows from Proposition~\ref{ext_surface}(ii).
\end{proof}

\section{Automorphisms of torsion sheaves on the projective line}
\label{automorphisms}

\noindent
In this section we prove that the automorphism group of a coherent sheaf on $\PP^1 = \PP^1(\CC)$ concentrated at a point
is of the form $U \rtimes P$, where $U$ is unipotent and $P$ is a parabolic subgroup of a general linear group.
This result will be used in the proof of Theorem~\ref{euler_fibers}.

Let $\CC[\zeta]$ be the polynomial ring over $\CC$ in one variable
and consider the principal ideal domain $R = \CC[\zeta]_{(\zeta)}^{}$.
Let $\Z$ be an $R$-module of finite dimension over $\CC$.
The structure theorem for finitely generated modules over a principal ideal domain tells us that
\[
\Z \simeq \bigoplus_{1 \le i \le n} R / (\zeta^{\nu_i}) \simeq \bigoplus_{1 \le i \le n} \CC[\zeta] / (\zeta^{\nu_i})
\]
for some integers $\nu_1^{} \ge \dots \ge \nu_n^{} > 0$ and $n \ge 1$.
Let $U$ and $P$ be the kernel, respectively, the image of the canonical morphism of algebraic groups
\begin{alignat*}{2}
\Aut_R^{}(\Z) & \lra \Aut_R^{}(\Z / \zeta \Z) && \simeq \GL(n, \CC). \\
\intertext{Let $\bar{U}$ and $\bar{P}$ be the kernel, respectively, the image of the canonical morphism of algebraic groups}
\Aut_R^{} (\Z) & \lra \Aut_R^{}(\Ann_{\Z}^{}(\zeta)) && \simeq \GL(n, \CC).
\end{alignat*}

\begin{remark}
\label{duality}
Dualising, i.e.\ applying $\Hom_\CC^{}(-, \CC)$ to the exact sequence of $R$-modules
\[
0 \lra \Ann_\Z^{}(\zeta) \lra \Z \overset{\cdot \zeta}{\lra} \Z \lra \Z / \zeta \Z \lra 0
\]
we obtain the exact sequence of $R$-modules
\[
0 \lra (\Z / \zeta \Z)^* \lra \Z^* \overset{\cdot \zeta}{\lra} \Z^* \lra (\Ann_{\Z}^{}(\zeta))^* \lra 0.
\]
Thus, we have a canonical isomorphism of $R$-modules $(\Ann_\Z^{}(\zeta))^* \simeq \Z^* / \zeta \Z^*$.
Fix a (non-canonical) isomorphism of $R$-modules $\varphi \colon \Z \to \Z^*$.
We have a commutative diagram of morphisms of algebraic groups
\[
\xymatrix
{
\Aut_R^{}(\Z) \ar[r] \ar[d]_-{\delta_1^{}}^-\simeq & \Aut_R^{}(\Ann_\Z^{}(\zeta)) \ar[d]^-{\delta_2^{}}_-\simeq \\
\Aut_R^{}(\Z^*) \ar[r] \ar[d]_-{\varphi_1^{}}^-\simeq & \Aut_R^{}(\Z^* / \zeta \Z^*) \ar[d]^-{\varphi_2^{}}_-\simeq \\
\Aut_R^{}(\Z) \ar[r] & \Aut_R^{}(\Z / \zeta \Z)
}
\]
in which $\delta_1^{}$ and $\delta_2^{}$ are the dualising isomorphisms, $\delta_1^{}(g)(h) = h \circ g^{-1}$,
while $\varphi_1^{}$ and $\varphi_2^{}$ are the isomorphisms induced by $\varphi$.
Since $\varphi_1^{} \circ \delta_1^{}(\bar{U}) = U$ and $\varphi_2^{} \circ \delta_2^{}(\bar{P}) = P$,
we obtain (non-canonical) isomorphisms of algebraic groups $U \simeq \bar{U}$, respectively, $P \simeq \bar{P}$.
\end{remark}

\begin{remark}
\label{hom_additive}
Let $\Z_1^{}$ and $\Z_2^{}$ be $R$-modules of finite dimension over $\CC$.
Then $(\Hom_R^{}(\Z_1^{}, \Z_2^{}), +)$ is a finite product of copies of $(\CC, +)$.
It is enough to verify this in the case when $\Z_1^{} = \CC[\zeta] / (\zeta^{\nu_1})$ and $\Z_2^{} = \CC[\zeta] / (\zeta^{\nu_2})$:
\[
(\Hom_R^{}(\Z_1^{}, \Z_2^{}), +) \simeq
\begin{cases}
\phantom{\zeta^{\nu_2 - \nu_1}} (\Z_2^{}, +) \simeq \prod_{\nu_2} (\CC, +) & \text{if $\nu_1^{} \ge \nu_2^{}$}, \\
(\zeta^{\nu_2 - \nu_1} \Z_2^{}, +) \simeq \prod_{\nu_1} (\CC, +) & \text{if $\nu_1^{} < \nu_2^{}$}.
\end{cases}
\]
\end{remark}

\begin{proposition}
\label{unipotent}
Let $U$ and $\bar{U}$ be the subgroups of $\Aut_R^{}(\Z)$ introduced above.
We claim that:
\begin{enumerate}
\item[\emph{(i)}] $U$ is unipotent;
\item[\emph{(ii)}] $\bar{U}$ is unipotent.
\end{enumerate}
\end{proposition}

\begin{proof}
If $\zeta \Z = \{ 0 \}$, then $U = \{ 1 \}$, so $U$ is unipotent.
To prove (i) we perform induction on $\dim_\CC^{} \Z$.
If $\dim_\CC^{} \Z = 1$, then $\zeta \Z = \{ 0 \}$, hence $U$ is unipotent.
Assume that $\dim_\CC^{} \Z > 1$ and that $\zeta \Z \neq \{ 0 \}$.
Some power of $\zeta$ annihilates $\Z$, hence there is a largest integer $\nu \ge 1$ such that $\zeta^\nu \Z \neq \{ 0 \}$.
By the induction hypothesis, the kernel $U'$ of the morphism of algebraic groups
\[
\Aut_R^{}(\Z / \zeta^\nu \Z) \lra \Aut_R^{}(\Z / \zeta \Z)
\]
is unipotent.
Being an algebraic subgroup of a unipotent group over $\CC$, the image of the morphism $U \to U'$ is unipotent.
It remains to show that the kernel of this morphism, denoted $U''$, is unipotent, as well.
Consider the injective algebraic map
\[
\iota \colon U'' \lra \Hom_R^{}(\Z, \zeta^\nu \Z) \quad \text{given by} \quad \iota(u)(z) = u(z) - z.
\]
Given $u_1^{}$, $u_2^{} \in U''$ and $z \in \Z$, we have the relations
\begin{align*}
(u_1^{} \circ u_2^{}) (z) & = \iota(u_1^{})(u_2^{}(z)) + u_2^{}(z) \\
& = \iota(u_1^{})(\iota(u_2^{})(z) + z) + \iota(u_2^{})(z) + z \\
& = \iota(u_1^{})(\iota(u_2^{})(z)) + \iota(u_1^{})(z) + \iota(u_2^{})(z) + z.
\end{align*}
By construction, $\iota(u_2^{})(z) \in \zeta^\nu \Z$, hence $\iota(u_1^{})(\iota(u_2^{})(z)) \in \zeta^{2 \nu} \Z = \{ 0 \}$.
Thus, for any $z \in \Z$,
\[
(u_1^{} \circ u_2^{})(z) = \iota(u_1^{})(z) + \iota(u_2^{})(z) + z,
\quad \text{forcing} \quad \iota(u_1^{} \circ u_2^{}) = \iota(u_1^{}) + \iota(u_2^{}).
\]
We deduce that $U''$ is isomorphic to an algebraic subgroup of $(\Hom_R^{}(\Z, \zeta^\nu \Z), + )$.
By Remark~\ref{hom_additive}, the latter is unipotent, forcing $U''$ to be unipotent, as well.
This concludes the proof of (i).
Part (ii) follows from the isomorphism $U \simeq \bar{U}$ of Remark~\ref{duality}.
\end{proof}

\begin{proposition}
\label{semidirect}
We adopt the above assumptions and notations.
We make the following claims:
\begin{enumerate}
\item[\emph{(i)}]
The morphism $\pi \colon \Aut_R^{}(\Z) \to P$ has a section, hence $\Aut_R^{}(\Z) \simeq U \rtimes P$.
Moreover, $P$ is a parabolic subgroup of $\GL(n, \CC)$.
\item[\emph{(ii)}]
The morphism $\bar{\pi} \colon \Aut_R^{}(\Z) \to \bar{P}$ has a section, hence $\Aut_R^{}(\Z) \simeq \bar{U} \rtimes \bar{P}$.
Moreover, $\bar{P}$ is a parabolic subgroup of $\GL(n, \CC)$.
\end{enumerate}
\end{proposition}

\begin{proof}
Take $g = (g_{ij}^{})_{1 \le i, j \le n}^{} \in \Aut_R^{}(\Z)$,
where $g_{ij}^{} \in \Hom_R^{}(\CC[\zeta] / (\zeta^{\nu_j}), \, \CC[\zeta] / (\zeta^{\nu_i}))$.
We have the relations
\[
g_{ij}^{}(1) =
\begin{cases}
a_{ij}^{} + \zeta b_{ij}^{} & \text{if $\nu_i^{} \le \nu_j^{}$}, \\
\zeta^{\nu_i - \nu_j} b_{ij}^{} & \text{if $\nu_i^{} > \nu_j^{}$},
\end{cases}
\qquad \text{where} \quad a_{ij}^{} \in \CC, \quad b_{ij}^{} \in \CC[\zeta] / (\zeta^{\nu_i}).
\]
Set $a_{ij}^{} = 0$ if $\nu_i^{} > \nu_j^{}$.
Note that $\pi(g) = (a_{ij}^{})_{1 \le i, j \le n}^{}$, so $P$ is contained in the parabolic subgroup
\[
P' = \{ a' = (a_{ij}')_{1 \le i, j \le n}^{} \mid a_{ij}' = 0 \ \text{if} \ \nu_i^{} > \nu_j^{} \} \le \GL(n, \CC).
\]
Given $a'$ as above, we construct $g' \in \End_R^{}(\Z)$ by setting $g_{ij}'(1) = a_{ij}'$.
By Nakayama's lemma, $g'$ is surjective, so $g' \in \Aut_R^{}(\Z)$.
This shows that $P = P'$.
The map $a' \mapsto g'$ is a section of $\pi$.
This concludes the proof of (i).
Part (ii) follows from (i) by duality, in view of Remark~\ref{duality}.
\end{proof}

\section{The topological Euler characteristic of the right and left fibers}
\label{topological}

\noindent
In the sequel we shall restrict our attention to moduli spaces of coherent systems on $X = \PP^1 \times \PP^1$
equipped with the polarization $\Osh(1, 1)$.
For a sheaf $\F$ on $X$ of dimension $1$ we have $\chern(\F) = (s, r)$ with integers $s$, $r \ge 0$ such that $\mult(\F) = r + s > 0$.
The expression $\Poly_{\F}^{}(x_1^{}, x_2^{}) = r x_1^{} + s x_2^{} + t$ will also be called the \emph{Hilbert polynomial} of $\F$.
Let $\M^\alpha((s, r), t)$ be the coarse moduli space of S-equivalence classes $\langle \Lambda \rangle$
of $\alpha$-semi-stable coherent systems $\Lambda = (\Gamma, \F)$ on $X$
having order $1$ and Hilbert polynomial $r x_1^{} + s x_2^{} + t$.
Let $\bar{r}$ be a positive integer.
We have a decomposition
\[
\M_X^\alpha(\bar{r}, t) = \bigsqcup_{\substack{r + s = \bar{r} \\ r, s \ge 0}} \M^\alpha((s, r), t)
\]
into closed subschemes, according to the first Chern class.
The entire discussion in section~\ref{right_left} remains valid if we replace $\M_X^\alpha(\bar{r}, t)$ with $\M^\alpha((s, r), t)$.
Thus, we can define a \emph{singular value} of $\alpha$ relative to the polynomial $r x_1^{} + s x_2^{} + t$.
The morphisms $\rho_{\alpha + \epsilon}^{}$ and $\rho_{\alpha - \epsilon}^{}$ in diagram~\eqref{wall-crossing}
are compatible with the above decomposition.
The canonical morphism
\begin{equation}
\label{gamma}
\bigsqcup \M^\alpha \big( (s', r'), \tfrac{(r' + s')(\alpha + t)}{r + s} - \alpha \big)^\st_\red
\times \M \big( (s - s', r - r'), \tfrac{(r + s - r' - s')(\alpha + t)}{r + s} \big)_\red^{}
\overset{\gamma}{\lra} \M^\alpha((s, r), t)^\pss
\end{equation}
induces a bijection between the sets of closed points.

Let $\pr_1^{}$, $\pr_2^{} \colon \PP^1 \times \PP^1 \to \PP^1$ be the projections onto the first and second component.
Let $p \in \PP^1$ be a closed point and let $\zeta$ be a local parameter of $\PP^1$ at $p$.
Consider the ring $R = \Osh_p^{} = \CC[\zeta]_{(\zeta)}^{}$.
Let $\mu$ be a positive integer and consider a partition of $\mu$ of the form
\[
N = \{ \nu_1^{} \ge \dots \ge \nu_n^{} > 0 \}
= \{ \nu_1^{} = \dots = \nu_{k_1}^{} > \nu_{k_1 + 1}^{} = \dots = \nu_{k_2}^{} > \dots > \nu_{k_{l - 1} + 1}^{} = \dots = \nu_{k_l}^{} > 0 \}.
\]
Let $d$ be an integer.
Consider the $R$-module ${\displaystyle \Z = \bigoplus_{1 \le i \le n} R / (\zeta^{\nu_i})}$
and the coherent $\Osh_{\PP^1 \times \PP^1}^{}$-module
\[
\Osh(d) \boxtimes \Z = \pr_1^* \Osh(d) \tensor \pr_2^* \Z.
\]
Let $\Lambda'$ be a coherent system on $\PP^1 \times \PP^1$.
Consider the linear maps
\[
\psi_\Rt^i \colon \Ext^1(\Lambda', \, \Osh(d) \boxtimes R / (\zeta^{\nu_i})) \lra \Ext^1(\Lambda', \, \Osh(d) \boxtimes R / (\zeta)) = V_\Rt^{}
\]
induced by the quotient morphisms $R / (\zeta^{\nu_i}) \to R / (\zeta)$.
The vector subspaces $\Img(\psi_\Rt^i) \subset V_\Rt^{}$ form a non-decreasing sequence because, for $i < j$,
$\psi_\Rt^i$ is the composite map
\[
\Ext^1(\Lambda', \, \Osh(d) \boxtimes R / (\zeta^{\nu_i})) \lra \Ext^1(\Lambda', \, \Osh(d) \boxtimes R / (\zeta^{\nu_j}))
\xrightarrow{\psi_\Rt^j} \Ext^1(\Lambda', \, \Osh(d) \boxtimes R / (\zeta)).
\]
For $1 \le j \le l$ denote $V_\Rt^j = \Img(\psi_\Rt^{k_j})$ and $v_\Rt^j = \dim_\CC^{} V_\Rt^j$.
Consider the flag variety of subspaces
\[
\FF(k_1^{}, \dots, k_l^{}; V_\Rt^{}) =
\{ (W_1^{}, \dots, W_l^{}) \mid W_j^{} \subset V_\Rt^{}, \ \dim_\CC^{} W_j^{} = k_j^{}, \ W_1^{} \subset W_2^{} \subset \dots \subset W_l^{} \}
\]
and the closed subvariety $\FF(k_1^{}, \dots, k_l^{}; v_\Rt^1, \dots, v_\Rt^l; V_\Rt^{})$
given by the additional condition that $W_j^{} \subset V_\Rt^j$ for $1 \le j \le l$.
The linear map
\[
\psi_\Rt^{} = \bigoplus_{1 \le i \le n} \psi_\Rt^i \colon \Ext^1(\Lambda', \, \Osh(d) \boxtimes \Z)
\lra \Ext^1(\Lambda', \, \Osh(d) \boxtimes \Z / \zeta \Z)
\]
is induced by the quotient morphism $\Z \to \Z / \zeta \Z$.
Notice that $\psi_\Rt^{}$ is $\Aut_R^{}(\Z)$-equivariant.
The open subsets
\[
\Ext^1(\Lambda', \, \Osh(d) \boxtimes \Z / \zeta \Z)_0^{}
= \{ (w_1^{}, \dots, w_n^{}) \mid w_i^{} \in V_\Rt^{} \ \text{are linearly independent} \}
\]
and
\[
\Ext^1(\Lambda', \, \Osh(d) \boxtimes \Z)_0^{} = \psi_\Rt^{-1} \Ext^1(\Lambda', \, \Osh(d) \boxtimes \Z / \zeta \Z)_0^{}
\]
are $\Aut_R^{}(\Z)$-invariant.
The above constructions will be used in the study of the right fibers.
For the study of the left fibers we consider the linear maps
\[
\psi_\Lt^i \colon \Ext^1(\Osh(d) \boxtimes R / (\zeta^{\nu_i}), \, \Lambda')
\lra \Ext^1(\Osh(d) \boxtimes R / (\zeta), \, \Lambda') = V_\Lt^{}
\]
induced by the inclusion morphisms $R / (\zeta) \xrightarrow{\cdot \zeta^{\nu_i - 1}} R / (\zeta^{\nu_i})$.
The vector subspaces $\Img(\psi_\Lt^i) \subset V_\Lt^{}$ form a non-decreasing sequence because, for $i < j$,
$\psi_\Lt^i$ is the composite map
\[
\Ext^1(\Osh(d) \boxtimes R / (\zeta^{\nu_i}), \, \Lambda') \lra \Ext^1(\Osh(d) \boxtimes R / (\zeta^{\nu_j}), \, \Lambda')
\xrightarrow{\psi_\Lt^j} \Ext^1(\Osh(d) \boxtimes R / (\zeta), \, \Lambda')
\]
induced by the inclusion morphisms of $R$-modules
\[
R / (\zeta) \xrightarrow{\cdot \zeta^{\nu_j - 1}} R / (\zeta^{\nu_j}) \xrightarrow{\cdot \zeta^{\nu_i - \nu_j}} R / (\zeta^{\nu_i}).
\]
For $1 \le j \le l$ denote $V_\Lt^j = \Img(\psi_\Lt^{k_j})$ and $v_\Lt^j = \dim_\CC^{} V_\Lt^j$.
Consider the flag variety of subspaces $\FF(k_1^{}, \dots, k_l^{}; V_\Lt^{})$
and the closed subvariety $\FF(k_1^{}, \dots, k_l^{}; v_\Lt^1, \dots, v_\Lt^l; V_\Lt^{})$
defined in the same way as for $V_\Rt^{}$.
The linear map
\[
\psi_\Lt^{} \colon \Ext^1(\Osh(d) \boxtimes \Z, \, \Lambda') \lra \Ext^1(\Osh(d) \boxtimes \Ann_\Z^{}(\zeta), \, \Lambda')
\]
induced by the inclusion morphism $\Ann_\Z^{}(\zeta) \to \Z$ can be identified with ${\displaystyle \oplus_{1 \le i \le n} \psi_\Lt^i}$,
once we have identified $\Ann_{R / (\zeta^{\nu_i})}(\zeta)$ with the image of the morphism
$R / (\zeta) \xrightarrow{\cdot \zeta^{\nu_i - 1}} R / (\zeta^{\nu_i})$.
As before, $\psi_\Lt^{}$ is $\Aut_R^{}(\Z)$-equivariant.
The open subsets
\[
\Ext^1(\Osh(d) \boxtimes \Ann_\Z^{}(\zeta), \, \Lambda')_0^{}
= \{ (w_1^{}, \dots, w_n^{}) \mid w_i^{} \in V_\Lt^{} \ \text{are linearly independent} \}
\]
and
\[
\Ext^1(\Osh(d) \boxtimes \Z, \, \Lambda')_0^{} = \psi_\Lt^{-1} \Ext^1(\Osh(d) \boxtimes \Ann_\Z^{}(\zeta), \, \Lambda')_0^{}
\]
are $\Aut_R^{}(\Z)$-invariant.
We fix an isomorphism $\Aut_R^{}(\Z) \simeq U \rtimes P$ as in Proposition~\ref{semidirect}(i).
We identify $P$ with the subgroup $\{ 1 \} \times P$.
We restrict the action of $\Aut_R^{}(\Z)$ on $\Ext^1(\Lambda', \, \Osh(d) \boxtimes \Z)$ to an action of $P$.
We fix an isomorphism $\Aut_R^{}(\Z) \simeq \bar{U} \rtimes \bar{P}$ as in Proposition~\ref{semidirect}(ii).
We identify $\bar{P}$ with the subgroup $\{ 1 \} \times \bar{P}$.
We restrict the action of $\Aut_R^{}(\Z)$ on $\Ext^1(\Osh(d) \boxtimes \Z, \, \Lambda')$ to an action of $\bar{P}$.
Given positive integers $k$ and $v$, we denote by $\GG(k, v)$ the Grassmann variety of $k$-dimensional subspaces of $\CC^v$.

\begin{proposition}
\label{modulo_P}
Relative to the above restricted actions we can construct the following geometric quotients:
\begin{enumerate}
\item[\emph{(i)}]
$\Ext^1(\Lambda', \, \Osh(d) \boxtimes \Z)_0^{} / P$ as an affine algebraic bundle
with base $\FF(k_1^{}, \dots, k_l^{}; v_\Rt^1, \dots, v_\Rt^l; V_\Rt^{})$ and  fiber $\Ker(\psi_\Rt^{})$.
This quotient has topological Euler characteristic
\[
\eulertop(\Ext^1(\Lambda', \, \Osh(d) \boxtimes \Z)_0^{} / P)
= \prod_{1 \le j \le l} \eulertop(\GG(k_j^{} - k_{j - 1}^{}, \, v_\Rt^j - k_{j - 1}^{}));
\]
\item[\emph{(ii)}]
$\Ext^1(\Osh(d) \boxtimes \Z, \, \Lambda')_0^{} / \bar{P}$ as an affine algebraic bundle
with base $\FF(k_1^{}, \dots, k_l^{}; v_\Lt^1, \dots, v_\Lt^l; V_\Lt^{})$ and fiber $\Ker(\psi_\Lt^{})$.
This quotient has topological Euler characteristic
\[
\eulertop(\Ext^1(\Osh(d) \boxtimes \Z, \, \Lambda')_0^{} / P)
= \prod_{1 \le j \le l} \eulertop(\GG(k_j^{} - k_{j - 1}^{}, \, v_\Lt^j - k_{j - 1}^{})).
\]
\end{enumerate}
\end{proposition}

\begin{proof}
To simplify notations, we write
\[
E_1^{} = \Ker(\psi_\Rt^{}), \quad E = \Ext^1(\Lambda', \, \Osh(d) \boxtimes \Z), \quad E_2^{} = \Img(\psi_\Rt^{}), \quad
F = \Ext^1(\Lambda', \, \Osh(d) \boxtimes \Z / \zeta \Z).
\]
It is well-known that the geometric quotient $F_0^{} / P$ exists and is isomorphic to $\FF(k_1^{}, \dots, k_l^{}; V_\Rt^{})$.
The quotient map is given by
\[
(w_1^{}, \dots, w_n^{}) \longmapsto (W_1^{}, \dots, W_l^{}), \quad \text{where} \quad
W_j^{} = \thespan \{ w_i^{} \mid 1 \le i \le k_j^{} \}.
\]
By construction, $E_{20}^{} = E_2^{} \cap F_0^{}$ is a closed and $P$-invariant subvariety of $F_0^{}$,
hence $E_{20}^{} / P$ exists and is the image of $E_{20}^{}$ in $F_0^{} / P$.
Explicitly, we have
\[
E_{20}^{} = \{ (w_1^{}, \dots, w_n^{}) \in F_0^{} \mid w_i^{} \in V_\Rt^j \ \text{if} \ i \le k_j^{} \},
\]
hence
\[
E_{20}^{} / P \simeq \FF(k_1^{}, \dots, k_l^{}; v_\Rt^1, \dots, v_\Rt^l; V_\Rt^{}).
\]
We fix an isomorphism $E \simeq E_1^{} \oplus E_2^{}$ such that $\psi_\Rt^{} = \pr_2^{}$.
Note that $\psi_\Rt^{}$ is $P$-equivariant, because it is $\Aut_R^{}(\Z)$-equivariant.
For $g \in P$, $e_1^{} \in E_1^{}$ and $e_2^{} \in E_2^{}$ we have the relations
\begin{align*}
\psi_\Rt^{}(g. (e_1^{}, 0)) & = g. \psi_\Rt^{}(e_1^{}, 0) = g. 0 = 0, \\
\psi_\Rt^{}(g. (0, e_2^{})) & = g. \psi_\Rt^{}(0, e_2^{}) = g. e_2^{},
\end{align*}
hence $g. (0, e_2^{}) = (\varrho(g, e_2^{}), g. e_2^{})$ for some algebraic map $\varrho \colon P \times E_2^{} \to E_1^{}$.
From the relations
\[
g. (e_1^{}, e_2^{}) = g. (e_1^{}, 0) + g. (0, e_2^{}) = (g. e_1^{}, 0) + (\varrho(g, e_2^{}), g. e_2^{}) = (g. e_1^{} + \varrho(g, e_2^{}), g. e_2^{})
\]
we see that the morphism $\psi_\Rt^{-1}(e_2^{}) \to \psi_\Rt^{-1}(g. e_2^{})$, $e \mapsto g. e$, is an affine automorphism of $E_1^{}$.
The action of $P$ on $E$ is, therefore, compatible with the structure of an affine algebraic bundle of $E$ with base $E_2^{}$ and fiber $E_1^{}$.
We can now apply the descent result \cite[Theorem 4.2.15]{huybrechts_lehn} to the affine algebraic bundle $E_0^{} |_{E_{20}}^{}$
in order to deduce that $E_0^{}$ descends to an affine algebraic bundle $Q_\Rt^{}$ with base $E_{20}^{} / P$ and fiber $E_1^{}$.
The key prerequisite condition for the descent result to work is that
$\Stab_P^{}(e_2^{})$ act trivially on $\psi_\Rt^{-1}(e_2^{})$, for every $e_2^{} \in E_{20}^{}$.
This condition is satisfied because the action of $P$ on $E_{20}^{}$ is free, being embedded in a free action of $\GL(n, \CC)$ on $F_0^{}$.
It is clear that $Q_\Rt^{}$ is the geometric quotient of $E_0^{}$ modulo $P$.
This concludes the proof of (i).
Analogously, the geometric quotient of $\Ext^1(\Osh(d) \boxtimes \Z, \, \Lambda')_0^{} / \bar{P}$
can be constructed as an affine algebraic bundle with fiber $\Ker(\psi_\Lt^{})$ and base
\[
\Img(\psi_\Lt^{})_0^{} / \bar{P} \simeq \FF(k_1^{}, \dots, k_l^{}; v_\Lt^1, \dots, v_\Lt^l; V_\Lt^{}).
\]
The latter occurs as a closed subvariety of
\[
\Ext^1(\Osh(d) \boxtimes \Ann_\Z^{}(\zeta), \, \Lambda')_0^{} / \bar{P} \simeq \FF(k_1^{}, \dots, k_l^{}; V_\Lt^{}).
\]
The formula for the Euler characteristic follows from the fact that $\FF(k_1^{}, \dots, k_l^{}; v_\Rt^1, \dots, v_\Rt^l; V_\Rt^{})$
is a tower of bundles with base $\GG(k_1^{}, v_\Rt^1)$
and fiber $\GG(k_j^{} - k_{j - 1}^{}, \, v_\Rt^j - k_{j - 1}^{})$ on the $j$-th floor.
\end{proof}

\begin{remark}
\label{local_sections}
Adopting the notations from the proof of Proposition~\ref{modulo_P}, we consider the Cartesian diagram
\[
\xymatrix
{
\Ext^1(\Lambda', \, \Osh(d) \boxtimes \Z)_0^{} \ar[r]^-{\tilde{\varphi}} \ar[d]_-{\psi_\Rt^{}}
& \Ext^1(\Lambda', \, \Osh(d) \boxtimes \Z)_0^{} / P = Q_\Rt^{} \ar[d] \\
E_{20}^{} \ar[r]^-{\varphi} & E_{20}^{} / P
}.
\]
in which $\varphi$ and $\tilde{\varphi}$ are the geometric quotient maps.
It is known that the quotient map $F_0^{} \to F_0^{} / P$ admits local sections, hence $\varphi$ admits local sections, as well.
It follows that $E_{20}^{} / P$ admits a finite decomposition $\Bf$
into locally closed reduced subschemes satisfying the following property:
for each $B \in \Bf$ there is a morphism of varieties $\xi_B^{} \colon B \to E_{20}^{}$
such that $\varphi \circ \xi_B^{}$ is the inclusion map of $B$.
Consider the cartesian diagram
\[
\xymatrix
{
Q_B^{} = Q_\Rt^{} |_B^{} \ar[r]^-{\tilde{\xi}_B^{}} \ar[d]_-{\psi_B^{}} & \Ext^1(\Lambda', \, \Osh(d) \boxtimes \Z)_0^{} \ar[d]^-{\psi_\Rt^{}} \\
B \ar[r]^-{\xi_B^{}} & E_{20}^{}
}.
\]
Notice that $\{ Q_B^{} \}_{B \in \Bf}^{}$ is a finite decomposition of $Q_\Rt^{}$ into locally closed reduced subschemes.
Recall the morphism $\pi$ from Proposition~\ref{semidirect}(i).
By construction, $U = \Ker(\pi)$ acts trivially on $E_{20}^{}$,
hence there is an induced algebraic action of $U$ on $Q_B^{}$ with respect to which $\tilde{\xi}_B^{}$ is equivariant.

We claim that two closed points $q$ and $q' \in Q_B^{}$ lie in the same $U$-orbit precisely if
$\tilde{\xi}_B^{}(q)$ and $\tilde{\xi}_B^{}(q')$ lie in the same $\Aut_R^{}(\Z)$-orbit.
Indeed, assume that $\tilde{\xi}_B^{}(q') = u g. \tilde{\xi}_B^{}(q)$ for some $u \in U$ and $g \in P$.
Then we have the relations
\begin{align*}
\xi_B^{}(\psi_B^{}(q')) & = \psi_\Rt^{}(\tilde{\xi}_B^{}(q')) = \psi_\Rt^{}(u g. \tilde{\xi}_B^{}(q)) = \pi(u g). \psi_\Rt^{}(\tilde{\xi}_B^{}(q)) \\
& = \pi(u) \pi(g). \psi_\Rt^{}(\tilde{\xi}_B^{}(q)) = \pi(g). \psi_\Rt^{}(\tilde{\xi}_B^{}(q)) = \pi(g). \xi_B^{}(\psi_B^{}(q)).
\end{align*}
We obtain the relations
\[
\psi_B^{}(q') = \varphi \circ \xi_B^{}(\psi_B^{}(q')) = \varphi(\pi(g). \xi_B^{}(\psi_B^{}(q))) = \varphi \circ \xi_B^{}(\psi_B^{}(q)) = \psi_B^{}(q).
\]
This shows that $\pi(g) \in \Stab_P^{}(\xi_B^{}(\psi_B^{}(q)))$.
We mentioned in the proof of Proposition~\ref{modulo_P} that all isotropy groups for the action of $P$ on $E_{20}^{}$ are trivial.
Thus $\pi(g) = 1$, forcing $g = 1$.
In conclusion, $\tilde{\xi}_B^{}(q') = u. \tilde{\xi}_B^{}(q) = \tilde{\xi}_B^{}(u. q)$ forcing $q' = u. q$, because $\tilde{\xi}_B^{}$ is injective.
\end{remark}

\begin{notation}
\label{Xi}
Let $\mu$ be  a positive integer.
We denote by $\Part(\mu)$ the set of partitions of $\mu$.
Consider a partition of $N$ of $\mu$ of the form
\[
N = \{ \nu_1^{} = \dots = \nu_{k_1}^{} > \nu_{k_1 + 1}^{} = \dots = \nu_{k_2}^{} > \dots > \nu_{k_{l - 1} + 1}^{} = \dots = \nu_{k_l}^{} > 0 \}.
\]
Let $\Lambda'$ be a coherent system on $\PP^1 \times \PP^1$.
Let $p \in \PP^1$ be a closed point and consider the line $L = \PP^1 \times \{ p \} \subset \PP^1 \times \PP^1$.
Let $d$ be an integer.
The expressions occurring on the r.h.s.\ in Proposition~\ref{modulo_P} depend only on $\Lambda'$, $L$, $d$ and $N$,
so we may define
\begin{alignat*}{3}
& \Xi_\Rt^{}(\Lambda', L, d, N) && = \prod_{1 \le j \le l} \eulertop(\GG(k_j^{} - k_{j - 1}^{}, \, v_\Rt^j - k_{j - 1}^{}))
&& = \prod_{1 \le j \le l} \binom{v_\Rt^j - k_{j - 1}^{}}{k_j^{} - k_{j - 1}^{}}, \\
& \Xi_\Lt^{}(\Lambda', L, d, N) && = \prod_{1 \le j \le l} \eulertop(\GG(k_j^{} - k_{j - 1}^{}, \, v_\Lt^j - k_{j - 1}^{}))
&& = \prod_{1 \le j \le l} \binom{v_\Lt^j - k_{j - 1}^{}}{k_j^{} - k_{j - 1}^{}}. \\
\intertext{Given a positive integer $v$, we write}
& \Xi(v, N) && = \prod_{1 \le j \le l} \eulertop(\GG(k_j^{} - k_{j - 1}^{}, \, v - k_{j - 1}^{}))
&& = \prod_{1 \le j \le l} \binom{v - k_{j - 1}^{}}{k_j^{} - k_{j - 1}^{}} \\
& && = \binom{v}{k_1^{}, \, k_2^{} - k_1^{}, \, \dots, \, k_l^{} - k_{l - 1}^{}, \, v - k_l^{}}.
\end{alignat*}
\end{notation}

\begin{proposition}{\cite[Proposition 10]{ballico_huh}}
\label{vartheta}
Let $\mu > 0$ and $d$ be integers.
Then there is an isomorphism
\[
\vartheta \colon \left| \Osh(0, \mu) \right| \lra \M((0, \mu), d \mu) \quad \text{given by} \quad
\vartheta \Big(\sum_{1 \le i \le m} \mu_i^{} L_i^{}\Big) = \Bigl< \bigoplus_{1 \le i \le m} \mu_i^{} \Osh_{L_i}^{}(d - 1, 0) \Bigr>.
\]
Here $L_i^{}$ are distinct lines in $\PP^1 \times \PP^1$ of degree $(0, 1)$
and $\{ \mu_1^{}, \dots, \mu_m^{} \}$ is a partition of $\mu$.
\end{proposition}

\noindent
The following lemma about the topological Euler characteristic is well-known,
cf.\ \cite[Section 2.1]{gulbrandsen}.

\begin{lemma}
\label{multiplicative}
Let $f \colon S \to T$ be a morphism of schemes of finite type over $\CC$.
Assume that there is $c \in \ZZ$ such that $\eulertop(f^{-1}(\tau)) = c$ for every closed point $\tau \in T$.
Then $\eulertop(S) = c \eulertop(T)$.
\end{lemma}

\begin{lemma}
\label{same_euler}
Let $U$ be a unipotent algebraic group over $\CC$.
Let $f \colon S \to T$ be a surjective morphism of algebraic varieties of finite type over $\CC$.
Assume that $S$ admits an algebraic action of $U$ such that,
for every $\tau \in T$, $f^{-1}(\tau)$ is a $U$-orbit.
Then $\eulertop(S) = \eulertop(T)$.
\end{lemma}

\begin{proof}
Recall that the topological Euler characteristic of any unipotent algebraic group over $\CC$ is $1$.
Recall that any algebraic subgroup of a unipotent algebraic group over $\CC$ is itself unipotent.
Thus, the topological Euler characteristic of every $U$-homogeneous space is $1$.
It follows that $\eulertop(f^{-1}(\tau)) = 1$ for every $\tau \in T$.
The conclusion follows from Lemma~\ref{multiplicative}
\end{proof}

\begin{theorem}
\label{euler_fibers}
Let $\alpha$ be a singular value relative to the polynomial $P(x_1^{}, x_2^{}) = r x_1^{} + s x_2^{} + t$.
Assume that, for some integers $d$ and $r'$, where $0 \le r' < r$, $\M^\alpha((s, r), t)^\pss$ contains the space
\[
\M^\alpha((s, r'), \, t - d(r - r'))^\st \times \M((0, r - r'), \, d(r - r')).
\]
Let $\mu_1^{}, \dots, \mu_m^{}$ be positive integers such that $\mu_1^{} + \dots + \mu_m^{} = r - r'$.
Consider distinct lines $L_i^{} = \PP^1 \times \{ p_i^{} \}$ in $\PP^1 \times \PP^1$.
Consider closed points
\[
\langle \Lambda' \rangle \in \M^\alpha((s, r'), \, t - d(r - r'))^\st
\]
and
\[
\langle \Theta \rangle = \Bigl< \bigoplus_{1 \le i \le m} \mu_i^{} \Osh_{L_i}^{}(d - 1, 0) \Bigr> \in \M((0, r - r'), \, d(r - r')).
\]
Then we have the following equations:
\begin{align*}
\tag{i}
\eulertop(\fiberR(\langle \Lambda' \rangle, \langle \Theta \rangle))
& = \sum_{N_1 \in \Part(\mu_1)} \dots \sum_{N_m \in \Part(\mu_m)}
\prod_{1 \le i \le m} \Xi_\Rt^{}(\Lambda', L_i^{}, d - 1, N_i^{}); \\
\tag{ii}
\eulertop(\fiberL(\langle \Lambda' \rangle, \langle \Theta \rangle))
& = \sum_{N_1 \in \Part(\mu_1)} \dots \sum_{N_m \in \Part(\mu_m)}
\prod_{1 \le i \le m} \Xi_\Lt^{}(\Lambda', L_i^{}, d - 1, N_i^{}).
\end{align*}
\end{theorem}

\begin{proof}
Recall, from Theorem~\ref{orbits}(i), that $\fiberR(\langle \Lambda' \rangle, \langle \Theta \rangle)$ can be decomposed
into constructible subsets of the form $\eta_\Rt^{}(\Ext^1(\Lambda', \Lambda'')^\ses)$, with $\Lambda'' \in \If(\Theta)$.
The set of isomorphism classes $\If(\Theta)$ is in a bijective correspondence
with the set of strings of partitions $\Nf = (N_i^{})_{1 \le i \le m}^{}$,
where $N_i^{} \in \Part(\mu_i^{})$ and $N_i^{} = \{ \nu_{ij}^{} \}_{1 \le j \le n_i}^{}$.
To $\Nf$ we associate the isomorphism class of
\[
\Lambda''_\Nf = \bigoplus_{1 \le i \le m} \Osh(d - 1) \boxtimes \Z_i^{},
\quad \text{where} \quad \Z_i^{} = \bigoplus_{1 \le j \le n_i} \Osh_{p_i}^{} / (\zeta_i^{\nu_{ij}}).
\]
Here $\zeta_i^{}$ is a local parameter of $\PP^1$ at $p_i^{}$.
Using the additivity of the Euler characteristic, we obtain
\begin{equation}
\label{additivity_1}
\eulertop(\fiberR(\langle \Lambda' \rangle, \langle \Theta \rangle))
= \sum_{N_1 \in \Part(\mu_1)} \dots \sum_{N_m \in \Part(\mu_m)} \eulertop( \eta_\Rt^{}(\Ext^1(\Lambda', \Lambda''_\Nf)^\ses)).
\end{equation}
We have described $\Ext^1(\Lambda', \Lambda''_\Nf)^\ses$ at Proposition~\ref{fibers}(i).
The sheaf $\Delta''$ must be of the form $\Osh_{L_i}^{}(d - 1, 0)$ for some index $1 \le i \le m$.
The surjective morphism $\Lambda''_\Nf \to \Delta''$ factors as follows:
\[
\Lambda''_\Nf \xrightarrow{\pr_i} \Osh(d - 1) \boxtimes \Z_i^{} \lra
\Osh(d - 1) \boxtimes \Z_i^{} / \zeta_i^{} \Z_i^{} \simeq \bigoplus_{n_i} \Osh_{L_i}^{}(d - 1, 0) \overset{a_i}{\lra} \Osh_{L_i}^{}(d - 1, 0).
\]
Here $a_i^{}$ is a non-zero row vector of length $n_i^{}$.
The morphism $\delta_\Rt^{}$ factors as follows:
\begin{multline*}
\Ext^1(\Lambda', \Lambda''_\Nf) \xrightarrow{\pr_i} \Ext^1(\Lambda', \, \Osh(d - 1) \boxtimes \Z_i^{})
\xrightarrow{\psi_\Rt} \Ext^1(\Lambda', \, \Osh(d - 1) \boxtimes \Z_i^{} / \zeta_i^{} \Z_i^{}) \\
\simeq \bigoplus_{n_i} \Ext^1(\Lambda', \, \Osh_{L_i}^{}(d - 1, 0)) \overset{a_i}{\lra} \Ext^1(\Lambda', \, \Osh_{L_i}^{}(d - 1, 0)).
\end{multline*}
According to Proposition~\ref{fibers}(i), an extension is semi-stable if and only if,
for all indices $i$ and for all non-zero row vectors $a_i^{}$, its image under the above composition is non-zero.
Equivalently,
\begin{align*}
\Ext^1(\Lambda', \Lambda''_\Nf)^\ses
& = \bigcap_{1 \le i \le m} \pr_i^{-1} \psi_\Rt^{-1} \Ext^1(\Lambda', \, \Osh(d - 1) \boxtimes \Z_i^{} / \zeta_i^{} \Z_i^{})_0^{} \\
& = \bigcap_{1 \le i \le m} \pr_i^{-1} \Ext^1(\Lambda', \, \Osh(d - 1) \boxtimes \Z_i^{})_0^{} \\
& \simeq \prod_{1 \le i \le m} \Ext^1(\Lambda', \, \Osh(d - 1) \boxtimes \Z_i^{})_0^{}.
\end{align*}
For each index $i$ we fix isomorphisms $\Aut_{\Osh_{p_i}}^{}(\Z_i^{}) \simeq U_i^{} \rtimes P_i^{}$
as in Proposition~\ref{semidirect}(i).
Note that
\[
\Aut(\Lambda''_\Nf) = \prod_{1 \le i \le m} \Aut_{\Osh_{p_i}}^{}(\Z_i^{}). \qquad
\text{Write} \quad U = \prod_{1 \le i \le m} U_i^{}, \quad P = \prod_{1 \le i \le m} P_i^{}.
\]
We identify $P$ with the subgroup $\{ 1 \} \times P \le \Aut(\Lambda''_\Nf)$.
The group $\Aut(\Lambda''_\Nf)$ acts diagonally on $\Ext^1(\Lambda', \Lambda''_\Nf)^\ses$.
According to Proposition~\ref{modulo_P}(i), relative to the restricted action of $P$, there exists a geometric quotient
\[
Q = \Ext^1(\Lambda', \Lambda''_\Nf)^\ses / P \simeq \prod_{1 \le i \le m} \Ext^1(\Lambda', \, \Osh(d - 1) \boxtimes \Z_i^{})_0^{} / P_i^{}.
\]
Moreover, we have the formula
\begin{equation}
\label{euler_Q}
\eulertop(Q) = \prod_{1 \le i \le m} \eulertop(\Ext^1(\Lambda', \, \Osh(d - 1) \boxtimes \Z_i^{})_0^{} / P_i^{})
= \prod_{1 \le i \le m} \Xi_\Rt^{}(\Lambda', L_i^{}, d - 1, N_i^{}).
\end{equation}
Denote by $\tilde{\varphi} \colon \Ext^1(\Lambda', \Lambda''_\Nf)^\ses \to Q$ the geometric quotient map modulo $P$.
Arguing as in Remark~\ref{local_sections}, we can show that $Q$ admits a finite decomposition $\{ Q_B^{} \}_{B \in \Bf}^{}$
into locally closed reduced subvarieties satisfying the following two properties.
Firstly, for each $B \in \Bf$, there is a morphism of varieties $\tilde{\xi}_B^{} \colon Q_B^{} \to \Ext^1(\Lambda', \Lambda''_\Nf)^\ses$
such that $\tilde{\varphi} \circ \tilde{\xi}_B^{}$ is the inclusion morphism of $Q_B^{}$.
Secondly, there is an algebraic action of $U$ on $Q_B^{}$ such that two closed points $q$ and $q' \in Q_B^{}$ lie in the same $U$-orbit
precisely if $\tilde{\xi}_B^{}(q)$ and $\tilde{\xi}_B^{}(q')$ lie in the same $\Aut(\Lambda''_\Nf)$-orbit.
According to Theorem~\ref{orbits}(i), the fibers of $\eta_\Rt^{}$ are precisely the $\Aut(\Lambda''_\Nf)$-orbits.
In particular, $\eta_\Rt^{}$ is $P$-equivariant.
By the universal property of a geometric quotient, there exists a morphism $\upsilon$ making the diagram
\[
\xymatrix
{
Q_B^{} \ar[r]^-{\tilde{\xi}_B^{}} \ar[dr]_-{\upsilon_B^{}} & \Ext^1(\Lambda', \Lambda''_\Nf)^\ses \ar[r]^-{\tilde{\varphi}} \ar[d]^-{\eta_\Rt^{}} & Q \ar[ld]^-{\upsilon} \\
& \fiberR(\langle \Lambda' \rangle, \langle \Theta \rangle)
}
\]
commute.
We claim that the fibers of $\upsilon_B^{} = \upsilon |_{Q_B}$ are precisely the $U$-orbits and that $Q_B^{}$ is a union of fibers of $\upsilon$.
Indeed, for a closed point $q \in Q_B^{}$ we have the relations
\begin{alignat*}{2}
\upsilon^{-1}(\upsilon(q)) & = \tilde{\varphi}(\Aut(\Lambda''_\Nf). \tilde{\xi}_B^{}(q)) \\
& = \tilde{\varphi}(\{ g u. \tilde{\xi}_B^{}(q) \mid u \in U, \ g \in P \}) \qquad & & \text{because $\Aut(\Lambda''_\Nf) \simeq U \rtimes P$} \\
& = \tilde{\varphi}(\{ u. \tilde{\xi}_B^{}(q) \mid u \in U \}) & & \text{because $\tilde{\varphi}$ is $P$-equivariant} \\
& = \tilde{\varphi}(\{ \tilde{\xi}_B^{}(u. q) \mid u \in U \}) & & \text{because $\tilde{\xi}_B^{}$ is $U$-equivariant} \\
& = U. q \subset Q_B^{}.
\end{alignat*}
Applying Lemma~\ref{same_euler} to the corestricted morphism $\upsilon_B^{} \colon Q_B^{} \to \upsilon(Q_B^{})$, we deduce that
\[
\eulertop(\upsilon(Q_B^{})) = \eulertop(Q_B^{}).
\]
Using the additivity of the topological Euler characteristic, we obtain the formula
\begin{equation}
\label{additivity_2}
\eulertop(\eta_\Rt^{}(\Ext^1(\Lambda', \Lambda''_\Nf)^\ses))
= \eulertop(\upsilon(Q)) = \sum_{B \in \Bf} \eulertop(\upsilon(Q_B^{}))
= \sum_{B \in \Bf} \eulertop(Q_B^{}) = \eulertop(Q).
\end{equation}
Part (i) of the proposition follows by combining formulas~\eqref{additivity_1}, \eqref{euler_Q} and \eqref{additivity_2}.
Part (ii) can be proved in a similar manner,
by exploring Theorem~\ref{orbits}(ii), Propositions~\ref{fibers}(ii), \ref{semidirect}(ii) and \ref{modulo_P}(ii),
and the analog of Remark~\ref{local_sections} for the quotient $\Ext^1(\Osh(d) \boxtimes \Z, \, \Lambda')_0^{} / \bar{P}$.
\end{proof}

\begin{proposition}
\label{Xi_calculation}
Let $\alpha$ be a singular value relative to the polynomial $P(x_1^{}, x_2^{}) = r x_1^{} + s x_2^{} + t$.
Assume that, for some integers $d$ and $r'$, where $0 \le r' < r$, $\M^\alpha((s, r), t)^\pss$ contains the space
\[
\M^\alpha((s, r'), \, t - d(r - r'))^\st \times \M((0, r - r'), \, d(r - r')).
\]
Consider a closed point
\[
\langle \Lambda' \rangle = \langle (\Gamma', \F') \rangle \in \M^\alpha((s, r'), \, t - d(r - r'))^\st.
\]
Consider a line $L = \PP^1 \times \{ p \}$ and a partition $N$ as in Notation~\ref{Xi}.
Let $\zeta$ be a local parameter of $\PP^1$ at $p$ and let $R = \Osh_p^{}$.
Then we have the following equations:
\begin{alignat*}{2}
\tag{i}
& \Xi_\Rt^{}(\Lambda', L, d - 1, N) && = \prod_{1 \le j \le l} \binom{v_\Rt^j - k_{j - 1}^{}}{k_j^{} - k_{j - 1}^{}}, \\
\tag{ii}
& \Xi_\Lt^{}(\Lambda', L, d - 1, N) && = \prod_{1 \le j \le l} \binom{v_\Lt^j - k_{j - 1}^{}}{k_j^{} - k_{j - 1}^{}},
\end{alignat*}
where
\begin{alignat*}{2}
v_\Rt^j & = \ext^1(\Lambda', \, \Osh(d - 1) \boxtimes R / (\zeta^{\nu_{k_j}}))
&& - \ext^1(\Lambda', \, \Osh(d - 1) \boxtimes R / (\zeta^{\nu_{k_j} - 1})), \\
v_\Lt^j & = \ext^1(\Osh(d - 1) \boxtimes R / (\zeta^{\nu_{k_j}}), \, \Lambda')
&& - \ext^1(\Osh(d - 1) \boxtimes R / (\zeta^{\nu_{k_j} - 1}), \, \Lambda') \\
& = \hom(\F', \, \Osh(d - 3) \boxtimes R / (\zeta^{\nu_{k_j}}))
&& - \hom(\F', \, \Osh(d - 3) \boxtimes R / (\zeta^{\nu_{k_j} - 1})) + s.
\end{alignat*}
\end{proposition}

\begin{proof}
Write $\nu = \nu_{k_j}^{}$.
From the exact sequence
\[
0 \lra \Osh(d - 1) \boxtimes R / (\zeta) \xrightarrow{\cdot \zeta^{\nu - 1}} \Osh(d - 1) \boxtimes R / (\zeta^\nu)
\lra \Osh(d - 1) \boxtimes R / (\zeta^{\nu - 1}) \lra 0
\]
we obtain the exact sequence
\begin{alignat*}{2}
& && \Hom(\Osh_L^{}(d - 1, 0), \Lambda') \\
& \lra \Ext^1(\Osh(d - 1) \boxtimes R / (\zeta^{\nu - 1}), \Lambda') \lra \Ext^1(\Osh(d - 1) \boxtimes R / (\zeta^\nu), \Lambda') \lra
&& \Ext^1(\Osh_L^{}(d - 1, 0), \Lambda').
\end{alignat*}
We apply Proposition~\ref{hom_vanishes}(ii) to $\Lambda'$ and to $\Lambda'' = (\{ 0 \}, \Osh_L^{}(d - 1, 0))$.
The hypothesis of Proposition~\ref{hom_vanishes} is satisfied
because $\Lambda'$ is $\alpha$-stable, $\Lambda''$ is semi-stable and $\p_\alpha^{}(\Lambda') = d = \p(\Lambda'')$.
We deduce that $\Hom(\Osh_L^{}(d - 1, 0), \, \Lambda') = \{ 0 \}$.
Recall that $v_\Lt^j$ is the dimension of the image of the last morphism in the above diagram.
We obtain the formula
\[
v_\Lt^j = \ext^1(\Osh(d - 1) \boxtimes R / (\zeta^\nu), \, \Lambda') - \ext^1(\Osh(d - 1) \boxtimes R / (\zeta^{\nu - 1}), \, \Lambda').
\]
We apply Proposition~\ref{ext_surface}(ii) to $\Lambda'$ and to $\Lambda_\nu'' = (\{ 0 \}, \Osh(d - 1) \boxtimes R / (\zeta^\nu))$.
The hypothesis of Proposition~\ref{ext_surface} is satisfied
because $\Lambda'$ is $\alpha$-stable, $\Lambda_\nu''$ is semi-stable and $\p_\alpha^{}(\Lambda') = \p(\Lambda_\nu'')$.
Thus,
\begin{align*}
\ext^1(\Osh(d - 1) \boxtimes R / (\zeta^\nu), \, \Lambda')
& = \hom(\F', \, \Osh(d - 3) \boxtimes R / (\zeta^\nu)) + \langle \chern(\Lambda'), \, \chern(\Osh(d - 1) \boxtimes R / (\zeta^\nu)) \rangle \\
& = \hom(\F', \, \Osh(d - 3) \boxtimes R / (\zeta^\nu)) + \langle (s, r'), (0, \nu) \rangle \\
& = \hom(\F', \, \Osh(d - 3) \boxtimes R / (\zeta^\nu)) + s \nu. \\
\intertext{Analogously, applying Proposition~\ref{ext_surface}(ii) to $\Lambda'$ and to $\Lambda_{\nu - 1}''$, we obtain the formula}
\ext^1(\Osh(d - 1) \boxtimes R / (\zeta^{\nu - 1}), \, \Lambda')
& = \hom(\F', \, \Osh(d - 3) \boxtimes R / (\zeta^{\nu - 1})) + s (\nu - 1).
\end{align*}
Thus,
\[
v_\Lt^j = \hom(\F', \, \Osh(d - 3) \boxtimes R / (\zeta^\nu)) - \hom(\F', \, \Osh(d - 3) \boxtimes R / (\zeta^{\nu - 1})) + s.
\]
Substituting the values of $v_\Lt^j$ into the expression $\Xi_\Lt^{}(\Lambda', L, d - 1, N)$, found at Notation~\ref{Xi}, concludes the proof of (ii).
The proof of part (i) is analogous and uses Proposition~\ref{hom_vanishes}(i).
\end{proof}

\noindent
In the remaining part of this section we shall examine the simplest case of Theorem~\ref{euler_fibers},
namely, the case in which $v_\Rt^j = \dim_{\CC}^{} V_\Rt^{}$ and $v_\Lt^j = \dim_{\CC}^{} V_\Lt^{}$ for all indices $j$.

\begin{proposition}
\label{Xi_simple}
We adopt the hypotheses and definitions of Proposition~\ref{Xi_calculation} and Notation~\ref{Xi}.
\begin{enumerate}
\item[\emph{(i)}]
We assume, in addition, that $d \ge 0$.
Then $\Xi_\Rt^{}(\Lambda', L, d - 1, N) = \Xi(d + s, N)$.
\item[\emph{(ii)}]
We assume, in addition, that $d \le 2$.
Then $\Xi_\Lt^{}(\Lambda', L, d - 1, N) = \Xi(s, N)$.
\end{enumerate}
\end{proposition}

\begin{proof}
We claim that, for each $1 \le j \le l$, $v_\Rt^j = \ext^1(\Lambda', \, \Osh_L^{}(d - 1, 0))$,
that is, $\psi_\Rt^{k_j}$ is surjective.
More generally, we claim that for every integer $\nu \ge 1$, the linear map
\[
\psi \colon \Ext^1(\Lambda', \, \Osh(d - 1) \boxtimes R / (\zeta^\nu)) \lra \Ext^1(\Lambda', \, \Osh(d - 1) \boxtimes R / (\zeta))
\]
induced by the quotient morphism $R / (\zeta^\nu) \to R / (\zeta)$ is surjective.
From the exact sequence
\[
0 \lra \Osh(d - 1) \boxtimes R / (\zeta^{\nu - 1}) \overset{\cdot \zeta}{\lra}
\Osh(d - 1) \boxtimes R / (\zeta^\nu) \lra \Osh(d - 1) \boxtimes R / (\zeta) \lra 0
\]
we obtain the exact sequence
\[
\Ext^1(\Lambda', \, \Osh(d - 1) \boxtimes R / (\zeta^\nu)) \overset{\psi}{\lra} \Ext^1(\Lambda', \, \Osh(d - 1) \boxtimes R / (\zeta))
\lra \Ext^2(\Lambda', \, \Osh(d - 1) \boxtimes R / (\zeta^{\nu - 1})).
\]
We now apply Proposition~\ref{ext_delPezzo}(i) with $\F'' = \Osh(d - 1) \boxtimes R / (\zeta^{\nu - 1})$.
The hypothesis of Proposition~\ref{ext_delPezzo}(i) is satisfied because $\F''$ is semi-stable,
$\p_{\alpha}^{}(\Lambda') = \p(\F'')$ and $\H^1(\F'') = \{ 0 \}$.
We deduce that $\Ext^2(\Lambda', \F'') = \{ 0 \}$.
This proves the claim.
We apply again Proposition~\ref{ext_delPezzo}(i), this time to the semi-stable sheaf $\Osh_L^{}(d - 1, 0)$.
We obtain the equations
\begin{align*}
v_\Rt^j = \ext^1(\Lambda', \, \Osh_L^{}(d - 1, 0)) & = \langle \chern(\Lambda'), \chern(\Osh_L^{}(d - 1, 0)) \rangle + \euler(\Osh_L^{}(d - 1, 0)) \\
& = \langle (s, r'), (0, 1) \rangle + d = s + d.
\end{align*}
Substituting $v_\Rt^j$ into the defining expression for $\Xi_\Rt^{}(\Lambda', L, d - 1, N)$,
yields part (i) of the proposition.
In order to prove part (ii), we exploit Proposition~\ref{ext_delPezzo}(ii) and we obtain the equations
$v_\Lt^j = \ext^1(\Osh_L^{}(d - 1, 0), \, \Lambda') = s$ for $1 \le j \le l$.
\end{proof}

\begin{proposition}
\label{Xi_combinations}
Let $\mu$ and $v$ be positive integers.
Then we have the equation
\[
\sum_{N \in \Part(\mu)} \Xi(v, N) = \binom{v + \mu - 1}{\mu}.
\]
\end{proposition}

\begin{proof}
Any $\mu$-combination with repetitions of $v$ gives rise to a partition $N \in \Part(\mu)$ as follows.
We represent the combination by the string $(o_1^{}, \dots, o_\mu^{})$.
We write $\{ o_1^{}, \dots, o_\mu^{} \} = \{ \bar{o}_1^{}, \dots, \bar{o}_n^{} \}$ with distinct $\bar{o}_i^{}$.
We take $N = \{ \nu_1^{}, \dots, \nu_n^{} \}$, where $\nu_i^{}$ is the number of times $\bar{o}_i^{}$ appears in the string.
Let us fix $N \in \Part(\mu)$ as in Notation~\ref{Xi}.
Each $\mu$-combination with repetitions of $v$ associated to $N$ is represented by a string of the form
\[
o = (\underbrace{o_{1_{}}^{}, \dots, o_1^{}}_{\nu_1},
\underbrace{o_{2_{}}^{}, \dots, o_2^{}}_{\nu_2}, \dots,
\underbrace{o_{k_l}^{}, \dots, o_{k_l}^{}}_{\nu_{k_l}}).
\]
There are $v! / (v - k_l^{})!$ such strings.
Two strings $o$ and $o'$ represent the same $\mu$-combination with repetitions of $v$ precisely if, for all indices $1 \le j \le l$,
\[
\{ o_{k_{j - 1} + 1}^{}, \dots, o_{k_j}^{} \} = \{ o'_{k_{j - 1} + 1}, \dots, o'_{k_j} \}.
\]
Thus, the equivalence class of $o$ has $(k_1^{})! (k_2^{} - k_1^{})! \dots (k_l^{} - k_{l - 1}^{})!$ elements.
It follows that the number of $\mu$-combinations with repetitions of $v$ associated to $N$ is
\[
\frac{v! / (v - k_l^{})!}{(k_1^{})! (k_2^{} - k_1^{})! \dots (k_l^{} - k_{l - 1}^{})!}
= \binom{v}{k_1^{}, \, k_2^{} - k_1^{}, \, \dots, \, k_l^{} - k_{l - 1}^{}, \, v - k_l^{}} = \Xi(v, N).
\]
The summation on the l.h.s.\ is the total number of $\mu$-combinations with repetitions of $v$.
\end{proof}

\begin{remark}
\label{discriminant}
Let $\mu$ be a positive integer.
Let $M = \{ \mu_1^{}, \dots, \mu_m^{} \}$ be a partition of $\mu$.
Consider the discriminant locus
\[
D_M^{} = \Big\{ \sum_{1 \le i \le m} \mu_i^{} L_i^{} \Bigm|
L_i^{} \subset \PP^1 \times \PP^1 \ \text{are mutually distinct lines of degree} \ (0, 1) \Big\} \subset \left| \Osh(0, \mu) \right|.
\]
Consider the open subvariety
\[
D_m^{} = \{ (p_1^{}, \dots, p_m^{}) \mid p_i^{} \in \PP^1 \ \text{are mutually distinct points} \} \subset (\PP^1)^m.
\]
Given $p = (p_1^{}, \dots, p_m^{}) \in D_m^{}$, consider the lines $L_i^{} = \PP^1 \times \{ p_i^{} \} \subset \PP^1 \times \PP^1$.
The map
\[
D_m^{} \lra D_M^{} \quad \text{given by} \quad p \longmapsto \sum_{1 \le i \le m} \mu_i^{} L_i^{}
\]
is a geometric quotient modulo the action of a subgroup  $\GM$ of the group of permutations of $m$ elements.
If $m \ge 3$, then $\eulertop(D_m^{}) = 0$, hence $\eulertop(D_M^{}) = 0$.
When $m = 2$ notice that
\[
D_{\{ \mu_1, \mu_2 \}}^{} \simeq
\begin{cases}
\PP^1 \times \PP^1 \smallsetminus \{ \text{diagonal} \} & \text{if} \ \mu_1^{} > \mu_2^{}, \\
\PP^2 \smallsetminus \{ \text{smooth conic} \} & \text{if} \ \mu_1^{} = \mu_2^{},
\end{cases}
\quad \text{hence} \quad
\eulertop(D_{\{ \mu_1, \mu_2 \}}^{}) =
\begin{cases}
2 & \text{if} \ \mu_1^{} > \mu_2^{}, \\
1 & \text{if} \ \mu_1^{} = \mu_2^{}.
\end{cases}
\]
When $m = 1$ notice that that $D_{\{ \mu \}}^{} \simeq \PP^1$, hence $\eulertop(D_{\{ \mu \}}^{}) = 2$.
\end{remark}

\begin{proposition}
\label{Xi_discriminant}
Let $\mu$ and $v$ be positive integers.
Then we have the equation
\[
\sum_{M = \{ \mu_1, \dots, \mu_m \} \in \Part(\mu)}
\eulertop(D_M^{}) \Big( \sum_{N_1 \in \Part(\mu_1)} \dots \sum_{N_m \in \Part(\mu_m)} \prod_{1 \le i \le m} \Xi(v, N_i^{}) \Big)
= \binom{2v + \mu - 1}{\mu}.
\]
\end{proposition}

\begin{proof}
For convenience, let us denote the l.h.s.\ by $\Xi(v, \mu)$.
Using Proposition~\ref{Xi_combinations}, we calculate:
\begin{align*}
\Xi(v, \mu) & = \sum_{M = \{ \mu_1, \dots, \mu_m \} \in \Part(\mu)}
\Big( \eulertop(D_M^{}) \prod_{1 \le i \le m} \Big( \sum_{N_i \in \Part(\mu_i)} \Xi(v, N_i^{}) \Big) \Big) \\
& = \sum_{M = \{ \mu_1, \dots, \mu_m \} \in \Part(\mu)}
\Big( \eulertop(D_M^{}) \prod_{1 \le i \le m} \rep{v}{\mu_i^{}} \Big).
\end{align*}
Substituting the formulas for $\eulertop(D_M^{})$ found at Remark~\ref{discriminant},
we obtain the following formulas:
\begin{align*}
\Xi(v, 2\mu + 1) & = 2 \rep{v}{2\mu + 1} + 2 \sum_{\mu_1 = \mu + 1}^{2\mu} \rep{v}{\mu_1^{}} \rep{v}{2\mu + 1 - \mu_1^{}}, \\
\Xi(v, 2\mu) \phantom{{} + 1}& = 2 \rep{v}{2\mu} + \rep{v}{\mu}^2 + 2 \sum_{\mu_1 = \mu + 1}^{2\mu - 1} \rep{v}{\mu_1^{}} \rep{v}{2\mu - \mu_1^{}}.
\end{align*}
These formulas can be written more compactly as
\[
\Xi(v, \mu) = 2\rep{v}{\mu} + \sum_{\nu = 1}^{\mu - 1} \rep{v}{\nu} \rep{v}{\mu - \nu}
= \sum_{\nu = 0}^{\mu} \rep{v}{\nu} \rep{v}{\mu - \nu}
= \rep{2v}{\mu}. \qedhere
\]
\end{proof}

\begin{theorem}
\label{euler_loci}
Let $\alpha$ be a singular value relative to the polynomial $P(x_1^{}, x_2^{}) = r x_1^{} + s x_2^{} + t$.
Assume that, for some integers $d$ and $r'$, where $0 \le r' < r$, $\M^\alpha((s, r), t)^\pss$ contains the space
\[
Y = \M^\alpha((s, r'), \, t - d(r - r'))^\st \times \M((0, r - r'), \, d(r - r')).
\]
\begin{enumerate}
\item[\emph{(i)}]
Assume, in addition, that $d \ge 0$.
Then we have the equation
\[
\eulertop(\fiberR Y) = \binom{2d + 2s + r - r' - 1}{r - r'} \eulertop(\M^\alpha((s, r'), \, t - d(r - r'))^\st).
\]
\item[\emph{(ii)}]
Assume, in addition, that $d \le 2$.
Then we have the equation
\[
\eulertop(\fiberL Y) = \phantom{2d + {}} \binom{2s + r - r' - 1}{r - r'} \eulertop(\M^\alpha((s, r'), \, t - d(r - r'))^\st).
\]
\end{enumerate}
\end{theorem}

\begin{proof}
Write $\M^\st = \M^\alpha((s, r'), \, t - d(r - r'))^\st$.
Recall the isomorphism $\vartheta$ from Proposition~\ref{vartheta}.
Given a partition $M = \{ \mu_1^{}, \dots, \mu_m^{} \} \in \Part(r - r')$, write $Y_M^{} = \M^\st \times \vartheta(D_M^{})$.
According to Theorem~\ref{euler_fibers}(i) and Proposition~\ref{Xi_simple}(i),
for any closed point $(\langle \Lambda' \rangle, \langle \Theta \rangle) \in Y_M^{}$ we have the equation
\begin{align*}
\eulertop(\fiberR(\langle \Lambda' \rangle, \langle \Theta \rangle))
= & \sum_{N_1 \in \Part(\mu_1)} \dots \sum_{N_m \in \Part(\mu_m)} \prod_{1 \le i \le m} \Xi(d + s, N_i^{}). \\
\intertext{This expression remains constant as the closed point varies in $Y_M^{}$,
so we may apply Lemma~\ref{multiplicative} to obtain the formula}
\eulertop(\fiberR Y_M^{})
= \eulertop(Y_M^{}) & \sum_{N_1 \in \Part(\mu_1)} \dots \sum_{N_m \in \Part(\mu_m)} \prod_{1 \le i \le m} \Xi(d + s, N_i^{}).
\end{align*}
This leads us to the equation
\[
\eulertop(\fiberR Y) = \eulertop(\M^\st) \sum_{M \in \Part(r - r')}
\eulertop(D_M^{}) \Big(\sum_{N_1 \in \Part(\mu_1)} \dots \sum_{N_m \in \Part(\mu_m)} \prod_{1 \le i \le m} \Xi(d + s, N_i^{})\Big).
\]
Employing Theorem~\ref{euler_fibers}(ii) and Proposition~\ref{Xi_simple}(ii),
we obtain a formula for $\eulertop(\fiberL Y)$ like the one above, but with $d + s$ replaced by $s$.
The theorem now follows from Proposition~\ref{Xi_discriminant}.
\end{proof}

\section{Homomorphisms of pure sheaves}
\label{homomorphisms}

\noindent
The computation of $\eulertop(\fiberL Y)$ from Theorem~\ref{euler_loci}(ii) is insufficient for our purposes.
In this section we shall further examine the expression
$\eulertop(\fiberL(\langle \Lambda' \rangle, \langle \Theta \rangle))$
(notation as in Theorem~\ref{euler_fibers}).
We work on $X = \PP^1 \times \PP^1$.
We restrict our attention to the case when $\Theta$ is stable,
that is, $\Theta = \Osh_L^{}(d - 1, 0)$.
We write $\Lambda' = (\Gamma', \F')$.
According to Proposition~\ref{fibers_stable}(ii) and Proposition~\ref{ext_surface}(ii),
\[
\eulertop(\fiberL(\langle \Lambda' \rangle, \langle \Theta \rangle))
= \ext^1(\Theta, \Lambda')
= \langle (s, r'), (0, 1) \rangle + \hom(\F', \Osh_L^{}(d - 3, 0))
= s + \hom(\F', \Osh_L^{}(d - 3, 0)).
\]
Thus, our task is to determine the dimension of $\Hom(\F', \Osh_L^{}(d - 3, 0))$.
We do this only in the case when $\Lambda'$ is $\infty$-stable.
As per Proposition~\ref{M_infinity}, in this case $\F'$ can be explicitly described.

Let $r$, $s$ and $l$ be non-negative integers.
Consider the flag Hilbert scheme $\Hilb((s, r), l)$ parametrizing pairs $(C, Z)$,
where $C \subset X$ is a curve given by a polynomial of degree $(s, r)$
and $Z \subset C$ is a finite scheme of length $l$.
According to \cite[Lemma 2.2]{advances}, there exists an extension
\begin{equation}
\label{O_C_Z}
0 \lra \Osh_C^{} \lra \F \lra \Extsh_{\Osh_X}^2(\Osh_Z^{}, \Osh_X^{}) \lra 0
\end{equation}
such that $\F$ is pure.
Moreover, $\F$ is unique up to an isomorphism.
We write $\F = \Osh_C^{}(Z)$.
The following proposition is a straightforward consequence of \cite[Proposition B.8]{pandharipande_thomas}.

\begin{proposition}
\label{M_infinity}
Let $r > 0$, $s > 0$ and $t \ge r + s - r s$ be integers.
Then there is an isomorphism
\[
\Hilb((s, r), \, t - r - s + r s) \lra \M^\infty((s, r), t) \quad \text{given by} \quad (C, Z) \longmapsto \langle (\H^0(\Osh_C^{}), \Osh_C^{}(Z)) \rangle.
\]
\end{proposition}

\begin{proposition}
\label{hom_pure}
Let $r \ge 0$, $s > 0$, $l \ge 0$ and $d$ be integers.
Consider $(C, Z) \in \Hilb((s, r), l)$.
Consider a line $L \subset X$ of degree $(0, 1)$.
\begin{enumerate}
\item[\emph{(i)}]
If $L \nsubseteq C$, then $\Hom(\Osh_C^{}(Z), \Osh_L^{}(d, 0)) = \{ 0 \}$.
\item[\emph{(ii)}]
We assume that $C = C' \cup L$, where $C'$ is a curve given by a polynomial of degree $(s, r - 1)$.
We claim that we have an exact sequence of the form
\[
0 \lra \Osh_L^{}(-s, 0) \lra \Osh_C^{} \lra \Osh_{C'}^{} \lra 0.
\]
We notice that the image of the composite morphism $\Osh_L^{}(-s, 0) \to \Osh_C^{} \to \Osh_Z^{}$
is the structure sheaf of a subscheme $Z_0^{} \subset L$ of length $k$.
We claim that
\[
\hom(\Osh_C^{}(Z), \Osh_L^{}(d, 0)) = \max \{ 0, \, d - k + 1 \}.
\]
\end{enumerate}
\end{proposition}

\begin{proof}
(i) From the short exact sequence~\eqref{O_C_Z} we obtain the exact sequence
\[
\Hom(\Extsh_{\Osh_X}^2(\Osh_Z^{}, \Osh_X^{}), \, \Osh_L^{}(d, 0))
\lra \Hom(\Osh_C^{}(Z), \, \Osh_L^{}(d, 0)) \lra \Hom(\Osh_C^{}, \, \Osh_L^{}(d, 0)).
\]
The space on the left vanishes because $\Extsh_{\Osh_X}^2(\Osh_Z^{}, \Osh_X^{})$ has finite support.
If $L \nsubseteq C$, then the space on the right also vanishes, forcing the middle space to vanish, as well.

\medskip

\noindent
(ii) The kernel of the quotient morphism $\Osh_C^{}(s, 0) \to \Osh_{C'}^{}(s, 0)$ is pure
and has Hilbert polynomial $P(x_1^{}, x_2^{}) = x_1^{} + 1$,
hence it is isomorphic to $\Osh_L^{}$.
Thus, we have an exact sequence as in the proposition.
Write $C = C'' \cup m L$, where $m \ge 1$ is an integer and $C''$ is a curve that does not contain $L$.
By the same argument, it is easy to show
that the kernel of the quotient morphism $\Osh_C^{}(s, 0) \to \Osh_{C''}^{}(s, 0)$ is isomorphic to $\Osh_{m L}^{}$.
We have the commutative diagram
\[
\xymatrix
{
0 \ar[r] & \Osh_L^{} \ar[r] \ar[d] & \Osh_C^{}(s, 0) \ar[r] \ar@{=}[d] & \Osh_{C'}^{}(s, 0) \ar[r] \ar[d] & 0 \\
0 \ar[r] & \Osh_{m L}^{} \ar[r] & \Osh_C^{}(s, 0) \ar[r] & \Osh_{C''}^{}(s, 0) \ar[r] & 0
}.
\]
Write $L = \PP^1 \times \{ p \}$ and choose a local parameter $\zeta$ of $\PP^1$ at $p$.
The image of the morphism $\Osh_L^{} \to \Osh_{m L}^{}$ is the annihilator of $\zeta$.

The \emph{dual} of a coherent sheaf $\F$ of dimension $1$ on $X$ is $\F^\dual = \Extsh_{\Osh_X}^1(\F, \omega_X^{})$.
According to \cite[Remark 4]{rendiconti}, any pure sheaf $\F$ of dimension $1$ on $X$ is reflexive,
i.e.\ the canonical morphism $\F \to \F^\ddual$ is an isomorphism.
Thus, if $\F$ and $\G$ are pure sheaves of dimension $1$ on $X$, then
\[
\Hom(\F, \G) \simeq \Hom(\G^\dual, \F^\dual).
\]
According to \cite[Equation 8]{advances}, $\Osh_C^{}(Z)^\dual \simeq \I_{Z, C}^{} \tensor \omega_C^{}$,
where $\I_{Z, C}^{} \subset \Osh_C^{}$ is the ideal sheaf of $Z$ in $C$
and $\omega_C^{} = \omega_X^{} \tensor \Osh_X^{}(C)|_C^{} \simeq \Osh_C^{}(s - 2, r - 2)$ is the dualising sheaf of $C$.
We deduce that
\[
\Hom(\Osh_C^{}(Z), \, \Osh_L^{}(d, 0)) \simeq \Hom(\Osh_L^{}(-d - 2, 0), \, \I_{Z, C}^{}(s - 2, r - 2))
\simeq \Hom(\Osh_L^{}(-d, 0), \, \I_{Z, C}^{}(s, 0)).
\]
Consider the commutative diagram with exact rows and columns:
\[
\xymatrix
{
& 0 \ar[d] & 0 \ar[d] & 0 \ar[d] \\
0 \ar[r] & \I \ar[r] \ar[d] & \I_{Z, C}^{}(s, 0) \ar[r] \ar[d] & \K \ar[d] \\
0 \ar[r] & \Osh_{m L}^{} \ar[r] \ar[d] & \Osh_C^{}(s, 0) \ar[r] \ar[d]  & \Osh_{C''}^{}(s, 0) \ar[r] \ar[d] & 0 \\
0 \ar[r] & \J \ar[r] & \Osh_Z^{} \ar[r] \ar[d] & \Osh_{Z \cap C''}^{} \ar[r] \ar[d] & 0 \\
& & 0 & 0
}.
\]
The space $\Hom(\Osh_L^{}(-d, 0), \, \Osh_{C''}^{}(s, 0))$ vanishes because $L \nsubseteq C''$,
hence also $\Hom(\Osh_L^{}(-d, 0), \K)$ vanishes.
From the first row of the diagram we obtain the isomorphism
\[
\Hom(\Osh_L^{}(-d, 0), \, \I_{Z, C}^{}(s, 0)) \simeq \Hom(\Osh_L^{}(-d, 0), \, \I).
\]
From the first column and the third row of the diagram we obtain the exact sequence
\[
0 \lra \Hom(\Osh_L^{}(-d, 0), \, \I) \lra \Hom(\Osh_L^{}(-d, 0), \, \Osh_{m L}^{}) \lra \Hom(\Osh_L^{}(-d, 0), \, \Osh_Z^{}).
\]
Any morphism $\Osh_L^{}(-d, 0) \to \Osh_{m L}^{}$ must factor through $\Osh_L^{}$ because $\zeta$ annihilates $\Osh_L^{}(-d, 0)$.
We obtain the exact sequence
\[
0 \lra \Hom(\Osh_L^{}(-d, 0), \, \I) \lra \Hom(\Osh_L^{}(-d, 0), \, \Osh_L^{}) \lra \Hom(\Osh_L^{}(-d, 0), \, \Osh_Z^{}).
\]
The kernel of the morphism $\Osh_L^{} \to \Osh_Z^{}$ is isomorphic to $\Osh_L^{}(-k, 0)$.
We conclude that
\[
\Hom(\Osh_C^{}(Z), \, \Osh_L^{}(d, 0)) \simeq \Hom(\Osh_L^{}(-d, 0), \, \Osh_L^{}(-k, 0)) \simeq \H^0(\Osh_L^{}(d - k, 0)). \qedhere
\]
\end{proof}

\begin{notation}
\label{nested}
Consider integers $r$, $s > 0$ and $l \ge k \ge 0$.
Consider the nested Hilbert scheme
\[
H \subset \Hilb((s, r), l) \times \Hilb((s, r - 1), \, l - k)
\]
parametrizing quadruples $(C, Z, C', Z')$ such that $Z \subset C$, $Z' \subset C'$, $C' \subset C$ and $Z' \subset Z$.
Consider the commutative diagram
\[
\xymatrix
{
\Osh_C^{} \ar[r]^-{\varphi} \ar[d] & \Osh_{C'}^{} \ar[d] \\
\Osh_Z^{} \ar[r]^-{\psi} & \Osh_{Z'}^{}
}.
\]
Let $\Hilb(k, (s, r), l) \subset H$ be the open subscheme given by the condition
that the induced morphism $\Kersh(\varphi) \to \Kersh(\psi)$ be surjective.
The set of closed points of $\Hilb(k, (s, r), l)$ can be identified with
\[
\{ (C, Z, L) \mid \deg C = (s, r), \ \length Z = l, \ Z \subset C, \ \deg L = (0, 1), \ L \subset C, \ \length Z_0^{} = k \}.
\]
\end{notation}

\begin{remark}
\label{k=0}
If $k = 0$, then $\Kersh(\psi) = \{ 0 \}$, hence $Z \subset C'$ and $\Hilb(0, (s, r), l) = H$.
Moreover,
\[
\Hilb(0, (s, r), l) \simeq \Hilb((s, r - 1), l) \times \left| \Osh(0, 1) \right|.
\]
\end{remark}

\begin{proposition}
\label{k=l}
We consider positive integers $r$, $s$ and $l$ such that $s + 1 \ge l$.
We claim that
\[
\eulertop(\Hilb(l, (s, r), l)) = 4.
\]
\end{proposition}

\begin{proof}
The set of closed points of $\Hilb(l, (s, r), l)$ can be identified with
\[
\{ (C, Z, L) \mid C = C' \cup L, \ \deg C' = (s, r - 1), \ \deg L = (0, 1), \ \length Z = l, \ Z \subset L \smallsetminus C' \}.
\]
Consider the morphism of schemes
\[
\varrho \colon \Hilb(l, (s, r), l) \lra \Hilb((0, 1), l) \quad \text{given by} \quad \varrho(C, Z, L) = (L, Z).
\]
If $(L, Z) \in \Hilb((0, 1), l)$, then $Z$ imposes $l$ linearly independent conditions on curves of degree $(s, r - 1)$.
We deduce that the reduced fiber $\varrho^{-1}(L, Z)$
can be identified with the complement in $\left| \Osh(s, r - 1) \right|$ of the union of $\length Z_\red^{}$ linearly independent hyperplanes.
We identify $\Hilb((0, 1), l)$ with $\left| \Osh(0, 1) \right| \times \PP^l$.
Under this identification, $\{ L \} \times \PP^l$ corresponds to the Hilbert scheme of $l$ points on $L$.
Consider $M = \{ \mu_1^{}, \dots, \mu_m^{} \} \in \Part(l)$.
Let $D_M^{} \subset \PP^l$ be the locally closed subset of subschemes $Z \subset \PP^1$ concentrated at $m$ distinct points,
of multiplicities $\mu_1^{}, \dots, \mu_m^{}$.
If $(L, Z) \in \left| \Osh(0, 1) \right| \times D_M^{}$, then $\varrho^{-1}(L, Z)$
can be identified with the complement $A_m^{}$ in $\left| \Osh(s, r - 1) \right|$ of the union of $m$ linearly independent hyperplanes.
In view of Lemma~\ref{multiplicative},
\[
\eulertop(\varrho^{-1}(\left| \Osh(0, 1) \right| \times D_M^{}))
= \eulertop(A_m^{}) \eulertop(\left| \Osh(0, 1) \right|) \eulertop(D_M^{})
= 2 \eulertop(A_m^{}) \eulertop(D_M^{}),
\]
hence
\[
\eulertop(\Hilb(l, (s, r), l))
= \sum_{M \in \Part(l)} 2 \eulertop(A_m^{}) \eulertop(D_M^{}).
\]
Finally, we substitute the values for $\eulertop(D_M^{})$, found at Remark~\ref{discriminant}, and for $\eulertop(A_m^{})$.
\end{proof}

\begin{proposition}
\label{k=1,l=2}
Let $r \ge 2$ and $s \ge 1$ be integers.
We claim that $\eulertop(\Hilb(1, (s, r), 2)) = 20$.
\end{proposition}

\begin{proof}
The set of closed points of $\Hilb(1, (s, r), 2)$ can be identified with
\begin{multline*}
\{ (C, Z, L) \mid C = C' \cup L, \ \deg C' = (s, r - 1), \ \deg L = (0, 1), \ \length Z = 2, \ Z \nsubseteq C', \\
Z \ \text{contains a point of} \ C' \}.
\end{multline*}
Let $\Hilb(2)$ denote the Hilbert scheme parametrizing subschemes of $X$ of length $2$.
Let $\Hilb_0^{}(2)$ denote the open subset of reduced subschemes.
Consider the morphism of schemes
\[
\varrho \colon \Hilb(1, (s, r), 2) \lra \Hilb(2) \quad \text{given by} \quad \varrho(C, Z, L) = Z.
\]
Every $Z \in \Hilb(2)$ imposes two linearly independent conditions on curves of degree $(s, r - 1)$.
The reduced fiber of $\varrho$ over $Z = \{ p_1^{}, p_2^{} \}$ is
\begin{align*}
\varrho^{-1}(Z) & = \{ (C' \cup L, L) \mid p_1^{} \in L \smallsetminus C', \ p_2^{} \in C' \}
\, \Sqcup \, \{ (C' \cup L, L) \mid p_2^{} \in L \smallsetminus C', \ p_1^{} \in C' \} \\
& \simeq \{ C' \mid p_1^{} \notin C', \ p_2^{} \in C' \} \, \Sqcup \, \{ C' \mid p_2^{} \notin C', \ p_1^{} \in C' \} \\
& \simeq (\PP^{s r + r - 2} \smallsetminus \PP^{s r + r - 3}) \, \Sqcup \, (\PP^{s r + r - 2} \smallsetminus \PP^{s r + r - 3}).
\end{align*}
We obtain $\eulertop(\varrho^{-1}(Z)) = 2$.
Assume now that $Z$ is a double point concentrated at $p$.
We have
\begin{align*}
\varrho^{-1}(Z) & = \{ (C, L) \mid C = C' \cup L, \ Z \subset C, \ Z \nsubseteq C', \ p \in C' \} \\
& = \{ (C, L) \mid C = C' \cup L, \ Z \subset C, \ Z \nsubseteq C', \ p \in C' \cap L \}. \\
\intertext{The condition that $Z$ be contained in $C$ is redundant because $C$ is singular at $p$.
Thus,}
\varrho^{-1}(Z) & = \{ (C, L) \mid C = C' \cup L, \ Z \nsubseteq C', \ p \in C' \cap L \} \\
& \simeq \{ C' \mid Z \nsubseteq C', \ p \in C' \} \simeq \PP^{s r + r - 2} \smallsetminus \PP^{s r + r - 3},
\end{align*}
hence $\eulertop(\varrho^{-1}(Z)) = 1$.
Using Lemma~\ref{multiplicative} and the additivity of the Euler characteristic, we get
\[
\eulertop(\Hilb(1, (s, r), 2)) = 2 \eulertop(\Hilb_0^{}(2)) + \eulertop(\Hilb(2) \smallsetminus \Hilb_0^{}(2)) = 20.
\qedhere
\]
\end{proof}

\begin{proposition}
\label{k=1,l=3}
Let $r \ge 3$ and $s \ge 2$ be integers.
We claim that $\eulertop(\Hilb(1, (s, r), 3)) = 76$.
\end{proposition}

\begin{proof}
The set of closed points of $\Hilb(1, (s, r), 3)$ can be identified with
\begin{multline*}
\{ (C, Z, L) \mid C = C' \cup L, \ \deg C' = (s, r - 1), \ \deg L = (0, 1), \ \length Z = 3, \ Z \nsubseteq C', \\
C' \ \text{contains a subscheme} \ Z' \subset Z \ \text{of length} \ 2 \}.
\end{multline*}
Let $\Hilb(3)$ be the Hilbert scheme parametrizing the subschemes $Z \subset X$ of length $3$.
For $0 \le i \le 3$ we consider the locally closed subschemes $\Hilb_i^{}(3) \subset \Hilb(3)$ defined as follows:
$\Hilb_0^{}(3)$ is the open subscheme parametrizing the reduced subschemes $Z \subset X$ of length $3$;
$\Hilb_1^{}(3)$ parametrizes the subschemes $Z \subset X$ that are a union of a simple point and a double point;
$\Hilb_2^{}(3)$ parametrizes the subschemes $Z \subset X$ of length $3$, concentrated at a point $p$,
such that $Z$ is contained in a curve which is smooth at $p$;
$\Hilb_3^{}(3)$ is the closed subscheme parametrizing the ordinary triple points of $X$.
We have the relations
\[
\eulertop(\Hilb_0^{}(3)) = 4, \qquad
\eulertop(\Hilb_1^{}(3)) = 24, \qquad
\eulertop(\Hilb_2^{}(3)) = 8, \qquad
\eulertop(\Hilb_3^{}(3)) = 4.
\]
Consider the morphism of schemes
\[
\varrho \colon \Hilb(1, (s, r), 3) \lra \Hilb(3) \quad \text{given by} \quad \varrho(C, Z, L) = Z.
\]
Every $Z \in \Hilb(3)$ imposes three linearly independent conditions on curves of degree $(s, r - 1)$.
This can be seen by using the minimal locally free resolution of the ideal sheaf $\I_Z^{}$.
If $Z$ is contained in a line, it is clear what the minimal resolution is.
If $Z$ is not contained in a line,
the minimal resolution of $\I_Z^{}$ is given at \cite[Lemma 2.1]{advances}.
Consider $Z = \{ p_1^{}, p_2^{}, p_3^{} \} \in \Hilb_0^{}(3)$.
The reduced fiber of $\varrho$ over $Z$ is
\begin{align*}
\varrho^{-1}(Z) & = \{ (C' \cup L, L) \mid p_1^{} \in L \smallsetminus C', \ p_2^{}, \, p_3^{} \in C' \} \, \Sqcup \\
& \phantom{{} = {}} \{ (C' \cup L, L) \mid p_2^{} \in L \smallsetminus C', \ p_3^{}, \, p_1^{} \in C' \} \,
\Sqcup \, \{ (C' \cup L, L) \mid p_3^{} \in L \smallsetminus C', \ p_1^{}, \, p_2^{} \in C' \} \\
& \simeq \{ C' \mid p_1^{} \notin C', \ p_2^{}, \, p_3^{} \in C' \} \,
\Sqcup \, \{ C' \mid p_2^{} \notin C', \ p_3^{}, \, p_1^{} \in C' \} \,
\Sqcup \, \{ C' \mid p_3^{} \notin C', \ p_1^{}, \, p_2^{} \in C' \} \\
& \simeq {} (\PP^{s r + r - 3} \smallsetminus \PP^{s r + r - 4}) \,
\Sqcup \, (\PP^{s r + r - 3} \smallsetminus \PP^{s r + r - 4}) \, \Sqcup \, (\PP^{s r + r - 3} \smallsetminus \PP^{s r + r - 4}).
\end{align*}
Thus, $\eulertop(\varrho^{-1}(Z)) = 3$.
Consider, $Z = \{ p_1^{} \} \cup Z_2^{} \in \Hilb_1^{}(3)$, where $Z_2^{}$ is a double point concentrated at $p_2^{}$.
Notice that $Z'$ must be either $\{ p_1^{}, p_2^{} \}$ or $Z_2^{}$, hence
\begin{align*}
\varrho^{-1}(Z) & = \{ (C' \cup L, L) \mid Z \subset C' \cup L, \ p_1^{} \in C', \ p_2^{} \in C', \ Z_2^{} \nsubseteq C' \} \\
& \phantom{{} = {}} \Sqcup \, \{ (C' \cup L, L) \mid p_1^{} \in L \smallsetminus C', \ Z_2^{} \subset C' \} \\
& = \{ (C' \cup L, L) \mid Z \subset C' \cup L, \ p_1^{} \in C', \ p_2^{} \in C' \cap L, \ Z_2^{} \nsubseteq C' \} \\
& \phantom{{} = {}} \Sqcup \, \{ (C' \cup L, L) \mid p_1^{} \in L \smallsetminus C', \ Z_2^{} \subset C' \}. \\
\intertext
{If $p_2^{} \in C' \cap L$, then $C' \cup L$ is singular at $p_2^{}$, hence $Z_2^{} \subset C' \cup L$, and hence the condition
that $Z$ be contained in $C' \cup L$ is redundant.
It follows that}
\varrho^{-1}(Z) & \simeq \{ C' \mid p_1^{}, \, p_2^{} \in C', \ Z_2^{} \nsubseteq C' \} \,
\Sqcup \, \{ C' \mid p_1^{} \notin C', \ Z_2^{} \subset C' \} \\
& \simeq (\PP^{s r + r - 3} \smallsetminus \PP^{s r + r - 4}) \, \Sqcup \, (\PP^{s r + r - 3} \smallsetminus \PP^{s r + r - 4}), \\
\intertext
{hence $\eulertop(\varrho^{-1}(Z)) = 2$.
Consider $Z \in \Hilb_2^{}(3)$ that is concentrated at $p$.
There is a unique subscheme $Z' \subset Z$ of length $2$ because $Z$ is contained in a curve that is smooth at $p$.
We have}
\varrho^{-1}(Z) & = \{ (C' \cup L, L) \mid Z \subset C' \cup L, \ Z \nsubseteq C', \ Z' \subset C' \} \\
& = \{ (C' \cup L, L) \mid Z \subset C' \cup L, \ Z \nsubseteq C', \ Z' \subset C', \ p \in C' \cap L \}. \\
\intertext
{We claim that the condition that $Z$ be contained in $C' \cup L$ is redundant.
We can choose an affine chart of $X$ around $p$ with coordinates $(x, y)$
such that $I(L) = (a x + b y)$ with $a$, $b \in \CC$,
and either $I(Z) = (x, y^3)$ or $I(Z) = (x + y^2, x y)$.
In both cases we have $I(Z') = (x, y^2)$.
Write $I(C') = (w)$ for some $w \in \CC [x, y]$.
By hypothesis, $w = w_1^{} x + w_2^{} y^2$ for some $w_1^{}$, $w_2^{} \in \CC [x, y]$.
We have}
w (a x + b y)
& = a w_1^{} x^2 + b w_1^{} x y + a w_2^{} x y^2 + b w_2^{} y^3 \in (x, y^3), \\
w (a x + b y)
& = a w_1^{} x(x + y^2) - a w_1^{} x y^2 + b w_1^{} x y \\
& \phantom{{} = {}} + a w_2^{} x y^2 + b w_2^{} (x + y^2) y - b w_2^{} x y \in (x + y^2, x y), \\
\intertext
{so $w (a x + b y) \in I(Z)$.
This proves the claim.
Removing the superfluous conditions, we see that}
\varrho^{-1}(Z) & = \{ (C' \cup L, L) \mid Z \nsubseteq C', \ Z' \subset C', \ p \in L \} \\
& \simeq \{ C' \mid Z \nsubseteq C', \ Z' \subset C' \} \simeq \PP^{s r + r - 3} \smallsetminus \PP^{s r + r - 4}, \\
\intertext
{hence $\eulertop(\varrho^{-1}(Z)) = 1$.
Finally, we consider an ordinary triple point $Z$ supported on $p$.
The subschemes $Z' \subset Z$ of length $2$ are parametrized by $\PP^1$.
Repeating the previous steps, we obtain}
\varrho^{-1}(Z) & \simeq \{ C' \mid Z \nsubseteq C', \ Z' \subset C' \}. \\
\intertext
{This is an algebraic bundle with fiber $\PP^{s r + r - 3} \smallsetminus \PP^{s r + r - 4}$ and base $\PP^1$, accounting for $Z'$,
hence $\eulertop(\varrho^{-1}(Z)) = 2$.
Applying Lemma~\ref{multiplicative} and the additivity of the Euler characteristic, yields}
\eulertop(\Hilb(1, (s, r), 3))
& = 3 \eulertop(\Hilb_0^{}(3)) +  2 \eulertop(\Hilb_1^{}(3)) + \eulertop(\Hilb_2^{}(3)) +  2 \eulertop(\Hilb_3^{}(3))
= 76. \qedhere
\end{align*}
\end{proof}

\section{The topological Euler characteristic of the moduli spaces}
\label{formulas}

\noindent
In this section we continue the study of moduli spaces on $X = \PP^1 \times \PP^1$.
Beginning with Proposition~\ref{singular}, we shall focus entirely on the case when the first Chern class is of the form $(2, r)$.

\begin{lemma}
\label{M_nonempty}
Let $r \ge 0$ and $s \ge 0$ be integers that are not both zero.
Let $t$ be an integer.
Let $\alpha > 0$ be a real number.
\begin{enumerate}
\item[\emph{(i)}]
If $\M^\alpha((s, r), t) \neq \emptyset$, then $t \ge r + s - r s$.
\item[\emph{(ii)}]
Assume that $r > 0$ and $s > 0$.
Then $\M^\alpha((s, r), t)^\st \neq \emptyset$ if and only if $t \ge r + s - r s$.
\end{enumerate}
\end{lemma}

\begin{proof}
At (i) we have quoted \cite[Lemma 4.3]{advances}.
Assume now that $r > 0$, $s > 0$ and $t \ge r + s - r s$.
Let $C \subset X$ be an irreducible curve given by a polynomial of degree $(s, r)$.
Let $Z \subset C$ be a subscheme of length $t - r - s + r s$.
According to Proposition~\ref{M_infinity}, $\Lambda = (\H^0(\Osh_C^{}), \Osh_C^{}(Z))$ gives a point in $\M^\infty((s, r), t)$.
Let $\alpha_1^{}, \dots, \alpha_m^{}$ be the singular values relative to the polynomial $r x_1^{} + s x_2^{} + t$.
Assume that $\Lambda$ is $(\alpha_i^{} + \epsilon)$-stable.
As noted supra diagram~\eqref{wall-crossing}, $\Lambda$ is $\alpha_i^{}$-semi-stable.
If $\Lambda$ were properly $\alpha_i^{}$-semi-stable, then $\Lambda$ would be an extension of $\Lambda'$ by $\Lambda''$,
as in Proposition~\ref{fibers}(i).
Thus, $C$ would be the union of the supports of $\Lambda'$ and $\Lambda''$,
contradicting the irreducibility of $C$.
We deduce that $\Lambda$ is $\alpha_i^{}$-stable.
By descending induction on $i$, we can show that $\Lambda$ is $\alpha_i^{}$-stable for all indices $1 \le i \le m$.
It follows that $\Lambda$ is $\alpha$-stable for all $\alpha \in (0, \infty)$.
\end{proof}

\begin{proposition}
\label{singular}
Consider integers $r \ge 1$ and $t \le 1$ satisfying the condition $r + t \ge 2$.
Then the set of singular values of $\alpha$ relative to the polynomial $P(x_1^{}, x_2^{}) = r x_1^{} + 2 x_2^{} + t$ is
\[
\Asing((2, r), t) = \Big\{ d(r + 2) - t \Bigm| \frac{t}{r + 2} < d \le r + t - 3, \ d \in \ZZ \Big\}.
\]
Moreover, for a singular value $\alpha_d^{} = d(r + 2) - t$, we have the decomposition
\begin{align*}
Y_d^{} & = \M^{\alpha_d}((2, r), t)^\pss = \bigsqcup_{r'} Y_{d, r'}^{} \quad \text{with} \quad
\frac{d r - t + 2}{d + 1} \le r' \le r - 1 \quad \text{and} \\
Y_{d, r'}^{} & = \M^{\alpha_d}((2, r'), \, t - d(r - r'))^\st \times \M((0, r - r'), \, d (r - r')).
\end{align*}
\end{proposition}

\begin{proof}
By definition, $\alpha \in \Asing((2, r), t)$ if $\M^\alpha((2, r), t)^\pss \neq \emptyset$.
Recalling the bijective morphism $\gamma$ from \eqref{gamma}, this is equivalent to saying
that there exist integers $r'$, $s'$ and $t'$ satisfying the following conditions:
\begin{gather*}
\tag{i}
0 \le r' \le r, \\
\tag{ii}
0 \le s' \le 2, \\
\tag{iii}
0 < r' + s' < r + 2, \\
\tag{iv}
\frac{\alpha + t}{r + 2} = \frac{\alpha + t'}{r' + s'} \iff \alpha = \frac{t(r' + s') - t'(r + 2)}{r + 2 - r' - s'}, \\
\tag{v}
\M^\alpha((s', r'), \, t')^\st \neq \emptyset, \\
\tag{vi}
\M((2 - s', r - r'), \, t - t') \neq \emptyset.
\end{gather*}
Since $\alpha > 0$, condition (iv) implies the inequality
\[
\tag{vii}
t' < \frac{t(r' + s')}{r + 2}.
\]
According to Lemma~\ref{M_nonempty}(i), condition (v) implies the inequality
\[
\tag{viii}
r' + s' - r' s' \le t'.
\]
Assume that $s' = 0$.
Using inequalities (viii), (vii), $t \le 1$ and (i), we obtain the inequalities
\[
r' < \frac{t r'}{r + 2} \le \frac{r'}{r + 2} \le \frac{r}{r + 2} < 1,
\]
forcing $r' = 0$, forcing $r' + s' = 0$.
This contradicts condition (iii).
Assume that $s' = 1$.
Using inequalities (viii), (vii), $t \le 1$ and (i), we obtain the inequalities
\[
1 < \frac{t(r' + 1)}{r + 2} \le \frac{r' + 1}{r + 2} \le \frac{r + 1}{r + 2} < 1.
\]
This yields a contradiction.
We have proved that $s' \ge 2$.
Analogously, inequalities (viii), (vii), $t \le 1$ and (ii) imply the inequality $r' \ge 2$.
From condition (ii) we deduce that $s' = 2$.
According to Lemma~\ref{M_nonempty}(ii),
the condition $\M^\alpha((2, r'), t')^\st \neq \emptyset$ is equivalent to the inequality $2 - r' \le t'$.
According to \cite[Proposition 10]{ballico_huh}, the condition $\M((0, r - r'), \, t - t') \neq \emptyset$
is equivalent to the equation $t - t' = d(r - r')$ for some $d \in \ZZ$.
Thus, $\alpha \in \Asing((2, r), t)$ if and only if there are integers $r'$ and $d$ satisfying the following conditions:
\begin{gather*}
\tag{ix}
2 \le r' \le r - 1, \\
\tag{x}
\alpha = d(r + 2) - t, \\
\tag{xi}
2 - r' \le t - d(r - r').
\end{gather*}
By hypothesis, $\alpha$ is positive, hence condition (x) implies the inequality
\[
\tag{xii}
\frac{t}{r + 2} < d.
\]
By hypothesis, $t \ge 2 - r$, hence $d > (2 - r) / (r + 2) > -1$.
Condition (xi) becomes the condition
\[
\tag{xiii}
\frac{d r - t + 2}{d + 1} \le r'.
\]
The first inequality in (ix) is superfluous because it is implied by conditions (xii) and (xiii).
Indeed, if $r' \le 1$, then condition (xiii) would imply the inequality $d(r - 1) + 1 \le t$, forcing $d = 0$ and $t = 1$.
This would contradict condition (xii).
From (ix) and (xiii) we obtain the inequality
\[
\tag{xiv}
d \le r + t - 3.
\]
In conclusion,
\[
\Asing((2, r), t) = \{ d(r + 2) - t \mid d \in \ZZ,\ d \ \text{satisfies conditions (xii) and (xiv)} \}.
\]
For a singular value $\alpha_d^{} = d(r + 2) - t$, we have the decomposition $Y_d^{} = \Sqcup_{r'} Y_{d, r'}^{}$,
in which $r'$ runs over all integers satisfying condition (xiii) and the condition $r' \le r - 1$.
\end{proof}

\noindent
Under the provisions of Proposition~\ref{singular},
and under the assumption that singular values exist,
the largest singular value relative to the polynomial $r x_1^{} + 2 x_2^{} + t$ is
\[
\alphamax((2, r), t) = (r + t - 3)(r + 2) - t.
\]
In the remaining part of this section we shall focus on the moduli spaces $\M^\alpha((2, r), t)$ with $r \ge 2$ and $t \le 1$.
We need to impose the condition $t \ge 2 - r$, otherwise, in view of Lemma~\ref{M_nonempty}(i),
the moduli space would be empty.
If $t = 2 - r$, then, by Proposition~\ref{singular}, there are no singular values of $\alpha$.
In view of Proposition~\ref{M_infinity}, for all $\alpha \in (0, \infty)$, we have the isomorphisms
\begin{equation}
\label{r+t=2}
\M^\alpha((2, r), 2 - r) = \M^\alpha((2, r), 2 - r)^\st \simeq \M^\infty((2, r), 2 - r) \simeq \left| \Osh(2, r) \right|.
\end{equation}
In the sequel we shall assume that $r + t \ge 3$, $r \ge 2$ and $t \le 1$.
Recall the sets $Y_{d, r'}^{}$ from Proposition~\ref{singular} and note that $d \ge 0$.
According to Theorem~\ref{euler_loci}, for $d = 0$, $1$ or $2$, we have the following formulas:
\begin{align}
\label{d=0}
\eulertop(\fiberL Y_{0, r'}^{}) - \eulertop(\fiberR Y_{0, r'}^{})
& = \left( \!\! \binom{3 + r - r'}{r - r'} - \binom{3 + r - r'}{r - r'} \!\! \right) \eulertop(\M^{-t}((2, r'), t)^\st) = 0, \\
\label{d=1}
\eulertop(\fiberL Y_{1, r'}^{}) - \eulertop(\fiberR Y_{1, r'}^{})
& = \left( \!\! \binom{3 + r - r'}{r - r'} - \binom{5 + r - r'}{r - r'} \!\! \right) \eulertop(\M^{r + 2 - t}((2, r'), \, t - r + r')^\st), \\
\label{d=2}
\eulertop(\fiberL Y_{2, r'}^{}) - \eulertop(\fiberR Y_{2, r'}^{})
& = \left( \!\! \binom{3 + r - r'}{r - r'} - \binom{7 + r - r'}{r - r'} \!\! \right) \eulertop(\M^{2 r + 4 - t}((2, r'), \, t - 2 r + 2 r')^\st).
\end{align}

\begin{proposition}
\label{0+}
Consider integers $r \ge 2$ and $t \le 1$ satisfying the condition $r + t \ge 3$.
Then
\[
\eulertop(\M^{0+}((2, r), t))
= \eulertop(\M^\infty((2, r), t)) + \sum_{1 \le d \le r + t - 3} (\eulertop(\fiberL Y_d^{}) - \eulertop(\fiberR Y_d^{})).
\]
\end{proposition}

\begin{proof}
If $t = 0$ or $1$, then, according to Proposition~\ref{singular}, the singular values are $\alpha_d^{}$ with $1 \le d \le r + t - 3$,
and the proposition follows from the additivity of the topological Euler characteristic.
If $t < 0$, then there is also the singular value $\alpha_0^{} = -t$.
Adding equation~\eqref{d=0} over all indices $r'$, we obtain the formula $\eulertop(\fiberL Y_0^{}) - \eulertop(\fiberR Y_0^{}) = 0$.
On the r.h.s.\ there is no contribution from $Y_0^{}$.
\end{proof}

\noindent
Assume now that $d \ge 3$.
At Theorem~\ref{euler_loci} we have an adequate formula for the Euler characteristic of the right fiber, but not for the left fiber.
At the next proposition we address this shortcoming in the simplest case,
namely the case in which $Y_d^{} = Y_{d, r -1}^{}$.
The reason why the case when $r - r' \ge 2$ is much more difficult is the following.
From Proposition~\ref{vartheta} it transpires that $\M((0, 1), d)$ consists only of stable points,
while $\M((0, r - r'), d(r - r'))$ consists only of properly semi-stable points if $r - r' \ge 2$.
In section~\ref{c=10} we shall discuss one situation in which $r - r' = 2$.

\begin{proposition}
\label{r-1}
Consider integers $r \ge 5$ and $t \le 1$ satisfying the condition $r + t \ge 6$.
Let $d$ be an integer satisfying the conditions $3 \le d \le r + t - 3$ and $r + t < 2 d + 4$.
Recall Notation~\ref{nested}.
Put $l = r + t - d - 3$ and $\bar{k} = \min \{ d - 3, l \}$.
Then we have the following equations:
\begin{align}
\tag{i}
\eulertop(\fiberR Y_d^{}) & = (2 d + 4) \eulertop(\Hilb((2, r - 1), l)); \\
\tag{ii}
\eulertop(\fiberL Y_d^{}) & = 4 \eulertop(\Hilb((2, r - 1), l)) + \sum_{0 \le k \le \bar{k}} (d - 2 - k) \eulertop(\Hilb(k, (2, r - 1), l)).
\end{align}
\end{proposition}

\begin{proof}
The inequality $r + t < 2 d + 4$ is equivalent to the inequality $r - 2 < (dr - t + 2) / (d + 1)$.
Thus, in Proposition~\ref{singular} we have only the possibility $r' = r - 1$, forcing $Y_d^{} = Y_{d, r- 1}^{}$.
The inequality $r + t < 2 d + 4$ is also equivalent to the inequality
\[
\alphamax((2, r - 1), t - d) = (r + t - d - 4)(r + 1) - t + d  < d(r + 2) - t = \alpha_d^{}.
\]
Thus, employing Propositions~\ref{M_infinity} and \ref{vartheta}, we obtain the description
\[
Y_d^{} = \M^\infty((2, r - 1), \, t - d) \times \M((0, 1), d) \simeq \Hilb((2, r - 1), l) \times \left| \Osh(0, 1) \right|.
\]
Formula (i) follows from Theorem~\ref{euler_loci}(i).
According to Propositions~\ref{fibers_stable}(ii) and \ref{ext_surface}(ii),
for a point $(C, Z, L) \in Y_d^{}$ we have the equation
\[
\eulertop(\fiberL(C, Z, L)) = 2 + \hom(\Osh_C^{}(Z), \, \Osh_L^{}(d - 3, 0)).
\]
Applying Proposition~\ref{hom_pure}, we obtain the equation
\[
\eulertop(\fiberL(C, Z, L)) =
\begin{cases}
d - k & \text{if} \ (C, Z, L) \in \Hilb(k, (2, r - 1), l) \ \text{for} \ k = 0, \dots, \bar{k}, \\
2 & \text{otherwise}.
\end{cases}
\]
Applying Lemma~\ref{multiplicative} and using the additivity of the Euler characteristic leads us to the formula
\[
\eulertop(\fiberL Y_d^{}) = 2  \eulertop(Y_d^{}) + \sum_{0 \le k \le \bar{k}} (d - 2 - k) \eulertop(\Hilb(k, (2, r - 1), l)).
\]
This proves formula (ii).
\end{proof}

\noindent
Under the provisions of Proposition~\ref{r-1} we have the formula
\begin{equation}
\label{d>2}
\eulertop(\fiberL Y_d^{}) - \eulertop(\fiberR Y_d^{}) = -2 d \eulertop(\Hilb((2, r - 1), l))
+ \sum_{0 \le k \le \bar{k}} (d - 2 - k) \eulertop(\Hilb(k, (2, r - 1), l)).
\end{equation}

\noindent
In section~\ref{c=10} we shall need a description of the set of stable coherent systems
corresponding to the largest singular value of $\alpha$.

\begin{proposition}
\label{stable_max}
We adopt the assumptions of Proposition~\ref{singular}.
We also assume that $r \ge 3$ and $r + t \ge 3$.
We let $\alpha = \alphamax((2, r), t)$.
We claim that
\[
\M^\alpha((2, r), t)^\st = \Hilb((2, r), r + t - 2) \smallsetminus \{ (C, Z) \mid C = C' \cup L, \ \deg C' = (2, r - 1), \ Z \subset L \}.
\]
\end{proposition}

\begin{proof}
We have $\alpha = \alpha_d^{}$ with $d = r + t - 3$.
We have $\dfrac{d r - t + 2}{d + 1} = r - 1$, hence $r' = r - 1$, forcing
\begin{align*}
Y_d^{} = Y_{d, r - 1}^{} & = \M^\alpha((2, r - 1), t - d)^\st \times \M((0, 1), d) \\
& = \M^\alpha((2, r - 1), 3 - r)^\st \times \M((0, 1), d) \simeq \left| \Osh(2, r - 1) \right| \times \M((0, 1), d).
\end{align*}
The last isomorphism follows from equation~\eqref{r+t=2}.
We can identify $\M^\alpha((2, r), t)^\st$ with
\[
\M^\infty((2, r), t) \smallsetminus \fiberR Y_d^{} \simeq \Hilb((2, r), r + t - 2) \smallsetminus \fiberR Y_d^{}.
\]
The last isomorphism follows from Proposition~\ref{M_infinity}.
It remains to identify $\fiberR Y_d^{}$.
Assume that $\Lambda = (\H^0(\Osh_C^{}), \Osh_C^{}(Z))$ gives a closed point in $\fiberR Y_d^{}$.
According to Proposition~\ref{fibers}(i), we have a non-split extension
\[
0 \lra \Lambda'' \lra \Lambda \lra \Lambda' \lra 0,
\]
where $\Lambda' = (\H^0(\Osh_{C'}^{}), \Osh_{C'}^{})$ gives a point in $\M^\infty((2, r - 1), 3 - r)$ and
$\Lambda'' = (\{ 0 \}, \Osh_L^{}(d - 1, 0))$ gives a point in $\M((0, 1), d)$.
By hypothesis, $\H^0(\Osh_C^{})$ maps isomorphically to $\H^0(\Osh_{C'}^{})$, hence $\Osh_C^{}$ maps surjectively to $\Osh_{C'}^{}$.
We obtain the exact diagram
\[
\xymatrix
{
& 0 \ar[d] & 0 \ar[d] \\
0 \ar[r] & \Osh_L^{}(-2, 0) \ar[r] \ar[d] & \Osh_C^{} \ar[r] \ar[d] & \Osh_{C'}^{} \ar[r] \ar@{=}[d] & 0 \\
0 \ar[r] & \Osh_L^{}(d - 1, 0) \ar[r] \ar[d] & \Osh_C^{}(Z) \ar[r] \ar[d] & \Osh_{C'}^{} \ar[r] & 0 \\
& \Extsh_{\Osh_X}^2(\Osh_Z^{}, \Osh_X^{}) \ar@{=}[r] \ar[d] & \Extsh_{\Osh_X}^2(\Osh_Z^{}, \Osh_X^{}) \ar[d] \\
& 0 & 0
}.
\]
The exactness of the first row follows from Proposition~\ref{hom_pure}(ii).
The second column is the exact sequence~\eqref{O_C_Z}.
The sheaf $\Osh_Z^{}$ is self-dual, i.e.\ $\Osh_Z^{} \simeq \Extsh_{\Osh_X}^2(\Osh_Z^{}, \Osh_X^{})$.
This follows from the fact that $\Extsh_{\Osh_X}^2(\Osh_Z^{}, \Osh_X^{})$ is a quotient of $\Osh_L^{}$.
We have proved that $C = C' \cup L$ and $Z \subset L$.
Conversely, we assume that $C = C' \cup L$, where $\deg C' = (2, r - 1)$,
and $Z \subset L$, where $\length Z = r + t - 2 = d + 1$.
Our aim is to show that $\Lambda = (\H^0(\Osh_C^{}), \Osh_C^{}(Z))$ gives a point in $\fiberR Y_d^{}$.
Let $\F$ be defined by the push-out diagram
\[
\xymatrix
{
& 0 \ar[d] & 0 \ar[d] \\
0 \ar[r] & \Osh_L^{}(-2, 0) \ar[r] \ar[d] & \Osh_C^{} \ar[r] \ar[d] & \Osh_{C'}^{} \ar[r] \ar@{=}[d] & 0 \\
0 \ar[r] & \Osh_L^{}(d - 1, 0) \ar[r] \ar[d] & \F \ar[r] \ar[d] & \Osh_{C'}^{} \ar[r] & 0 \\
& \Osh_Z^{} \ar@{=}[r] \ar[d] & \Osh_Z^{} \ar[d] \\
& 0 & 0
}.
\]
Both $\Osh_L^{}(d - 1, 0)$ and $\Osh_{C'}^{}$ are pure,
hence also $\F$ is pure.
Since $\Osh_Z^{}$ is self-dual,
the middle column is the exact sequence~\eqref{O_C_Z}.
From \cite[Lemma 2.2]{advances} we deduce that $\F \simeq \Osh_C^{}(Z)$.
We conclude that $\Lambda$ is an extension of $\Lambda'$ by $\Lambda''$, so it gives a point in $\fiberR Y_d^{}$.
\end{proof}

\subsection{The case when $r + t \le 9$}
\label{c<10}
In this subsection we examine the Euler characteristic of $\M^{0+}((2, r), \, c - r)$ for $3 \le c \le 9$ and $r \ge c - 1$.
We apply Proposition~\ref{0+} and we use induction on $c$.
If $d = 1$ or $2$, then $\eulertop(\fiberL Y_d^{}) - \eulertop(\fiberR Y_d^{})$ can be computed
using formulas~\eqref{d=1} and \eqref{d=2}
and substituting the values of $\eulertop(\M^{\alpha_d}((2, r'), c' - r')^\st)$ with $c' < c$.
If $d = 3$, then the hypotheses of Proposition~\ref{r-1} are satisfied, so we have formula~\eqref{d>2}.
The details of the calculations are straightforward, so we present only the results:
\begin{align*}
\eulertop(\M^{0+}((2, r), \, 3 - r)) & = \eulertop(\Hilb((2, r), 1)), \\
\eulertop(\M^{0+}((2, r), \, 4 - r)) & = \eulertop(\Hilb((2, r), 2)) \\
& \phantom{{} =} {} - \phantom{1} 2 \eulertop(\Hilb((2, r - 1), 0)), \\
\eulertop(\M^{0+}((2, r), \, 5 - r)) & = \eulertop(\Hilb((2, r), 3)) \\
& \phantom{{} = } {} - \phantom{1} 4 \eulertop(\Hilb((2, r - 1), 0)) \\
& \phantom{{} = } {} - \phantom{1} 2 \eulertop(\Hilb((2, r - 1), 1)), \\
\eulertop(\M^{0+}((2, r), \, 6 - r)) & = \eulertop(\Hilb((2, r), 4)) \\
& \phantom{{} =} {} - \phantom{1} 6 \eulertop(\Hilb((2, r - 1), 0)) + \phantom{1}  \eulertop(\Hilb(0, (2, r - 1), 0)) \\
& \phantom{{} =} {} - \phantom{1} 4 \eulertop(\Hilb((2, r - 1), 1)) \\
& \phantom{{} =} {} - \phantom{1} 2 \eulertop(\Hilb((2, r - 1), 2)) \\
& \phantom{{} =} {} + \phantom{11} \eulertop(\Hilb((2, r - 2), 0)), \\
\eulertop(\M^{0+}((2, r), \, 7 - r)) & = \eulertop(\Hilb((2, r), 5)) \\
& \phantom{{} =} {} - \phantom{1} 8 \eulertop(\Hilb((2, r - 1), 0)) + 2 \eulertop(\Hilb(0, (2, r - 1), 0)) \\
& \phantom{{} =} {} - \phantom{1} 6 \eulertop(\Hilb((2, r - 1), 1)) + \phantom{1} \eulertop(\Hilb(0, (2, r - 1), 1)) \\
& \phantom{{} =} {} - \phantom{1} 4 \eulertop(\Hilb((2, r - 1), 2)) \\
& \phantom{{} =} {} - \phantom{1} 2 \eulertop(\Hilb((2, r - 1), 3)) \\
& \phantom{{} =} {} + \phantom{1} 8 \eulertop(\Hilb((2, r - 2), 0)) \\
& \phantom{{} =} {} + \phantom{11} \eulertop(\Hilb((2, r - 2), 1)), \\
\eulertop(\M^{0+}((2, r), \, 8 - r)) & = \eulertop(\Hilb((2, r), 6)) \\
& \phantom{{} =} {} - 10 \eulertop(\Hilb((2, r - 1), 0)) + 3 \eulertop(\Hilb(0, (2, r - 1), 0)) \\
& \phantom{{} =} {} - \phantom{1} 8 \eulertop(\Hilb((2, r - 1), 1)) + 2 \eulertop(\Hilb(0, (2, r - 1), 1)) \\
& \phantom{{} =} {} - \phantom{1} 6 \eulertop(\Hilb((2, r - 1), 2)) + \phantom{1} \eulertop(\Hilb(1, (2, r - 1), 1)) \\
& \phantom{{} =} {} - \phantom{1} 4 \eulertop(\Hilb((2, r - 1), 3)) + \phantom{1} \eulertop(\Hilb(0, (2, r - 1), 2)) \\
& \phantom{{} =} {} - \phantom{1} 2 \eulertop(\Hilb((2, r - 1), 4)) \\
& \phantom{{} =} {} + 18 \eulertop(\Hilb((2, r - 2), 0)) - 2 \eulertop(\Hilb(0, (2, r - 2), 0)) \\
& \phantom{{} =} {} + \phantom{1} 8 \eulertop(\Hilb((2, r - 2), 1)) \\
& \phantom{{} =} {} + \phantom{11} \eulertop(\Hilb((2, r - 2), 2)), \\
\eulertop(\M^{0+}((2, r), \, 9 - r)) & = \eulertop(\Hilb((2, r), 7)) \\
& \phantom{{} =} {} - 12 \eulertop(\Hilb((2, r - 1), 0)) + 4 \eulertop(\Hilb(0, (2, r - 1), 0)) \\
& \phantom{{} =} {} - 10 \eulertop(\Hilb((2, r - 1), 1)) + 3 \eulertop(\Hilb(0, (2, r - 1), 1)) \\
& \phantom{{} =} {} - \phantom{1} 8 \eulertop(\Hilb((2, r - 1), 2)) + 2 \eulertop(\Hilb(1, (2, r - 1), 1)) \\
& \phantom{{} =} {} - \phantom{1} 6 \eulertop(\Hilb((2, r - 1), 3)) + 2 \eulertop(\Hilb(0, (2, r - 1), 2)) \\
& \phantom{{} =} {} - \phantom{1} 4 \eulertop(\Hilb((2, r - 1), 4)) + \phantom{1} \eulertop(\Hilb(1, (2, r - 1), 2)) \\
& \phantom{{} =} {} - \phantom{1} 2 \eulertop(\Hilb((2, r - 1), 5)) + \phantom{1} \eulertop(\Hilb(0, (2, r - 1), 3)) \\
& \phantom{{} =} {} + 40 \eulertop(\Hilb((2, r - 2), 0)) - 8 \eulertop(\Hilb(0, (2, r - 2), 0)) \\
& \phantom{{} =} {} + 18 \eulertop(\Hilb((2, r - 2), 1)) - 2 \eulertop(\Hilb(0, (2, r - 2), 1)) \\
& \phantom{{} =} {} + \phantom{1} 8 \eulertop(\Hilb((2, r - 2), 2)) \\
& \phantom{{} =} {} + \phantom{11} \eulertop(\Hilb((2, r - 2), 3)) \\
& \phantom{{} =} {} - \phantom{1} 4 \eulertop(\Hilb((2, r - 3), 0)).
\end{align*}

\subsection{The case when $r + t = 10$}
\label{c=10}
In this subsection we examine the topological Euler characteristic of $\M^{0+}((2, r), 10 - r)$ for $r \ge 9$.
We apply Proposition~\ref{0+}.
If $d = 1$ or $2$, then $\eulertop(\fiberL Y_d^{}) - \eulertop(\fiberR Y_d^{})$ can be computed
using formulas~\eqref{d=1} and \eqref{d=2}
and substituting the values of $\eulertop(\M^{\alpha_d}((2, r'), c' - r')^\st)$ with $c' < 10$.
If $d = 4$, then the hypotheses of Proposition~\ref{r-1} are satisfied, so we have formula~\eqref{d>2}.
It remains to compute $\eulertop(\fiberL Y_3^{}) - \eulertop(\fiberR Y_3^{})$.
We have $Y_3^{} = Y_{3, r - 1}^{} \Sqcup \, Y_{3, r - 2}^{}$ and $\alpha_3^{} = 4 r - 4$.
We have
\[
Y_{3, r - 1}^{} = \M^{4 r - 4}((2, r - 1), \, 7 - r)^\st \times \M((0, 1), 3) \simeq \M^{4 r - 4}((2, r - 1), \, 7 - r)^\st \times \left| \Osh(0, 1) \right|.
\]
We have $\alphamax((2, r - 1), 7 - r) = 4 r - 4 = \alpha_3^{}$, hence, in view of Proposition~\ref{stable_max},
\[
\M^{4 r - 4}((2, r - 1), \, 7 - r)^\st = \Hilb((2, r - 1), 4) \smallsetminus H_1^{},
\]
where
\begin{align*}
H_1^{} & = \{ (C, Z) \mid C = C' \cup L', \ \deg C' = (2, r - 2), \ \deg L' = (0, 1), \ \length Z = 4, \ Z \subset L' \} \\
& \simeq \left| \Osh(2, r - 2) \right| \times \Hilb((0, 1), 4) \simeq \left| \Osh(2, r - 2) \right| \times \PP^1 \times \PP^4.
\end{align*}
Recall from Notation~\ref{nested} the variety $\Hilb(0, (2, r - 1), 4)$.
Its set of closed points is identified with
\[
\{ (C, Z, L) \mid C = C'' \cup L, \ \deg C'' = (2, r - 2), \ \deg L = (0, 1), \ \length Z = 4, \ Z \subset C'' \}.
\]
Using B\'ezout's theorem, we can easily show that
\[
\Hilb(0, (2, r - 1), 4) \cap (H_1^{} \times \left| \Osh(0, 1) \right|) = H_2^{},
\]
where
\begin{align*}
H_2^{} & = \{ (C, Z, L) \mid C = C_0^{} \cup L \cup L', \ \deg C_0^{} = (2, r - 3), \ \deg L = \deg L' = (0, 1), \ Z \subset L' \} \\
& \simeq \left| \Osh(2, r - 3) \right| \times \Hilb((0, 1), 4) \times \left| \Osh(0, 1) \right|
\simeq \left| \Osh(2, r - 3) \right| \times \PP^1 \times \PP^4 \times \left| \Osh(0, 1) \right|.
\end{align*}
According to Theorem~\ref{euler_loci}(i),
\[
\eulertop \big(\fiberR Y_{3, r - 1}^{}\big) = 10 \eulertop (\M^{4 r - 4}((2, r - 1), \, 7 - r)^\st).
\]
Arguing as in the proof of Proposition~\ref{r-1}(ii), we obtain the equation
\begin{align*}
\eulertop(\fiberL Y_{3, r - 1}^{}) & = 2  \eulertop(Y_{3, r - 1}^{}) + \eulertop(\Hilb(0, (2, r - 1), 4)) - \eulertop(H_2^{}) \\
& = 4 \eulertop(\M^{4 r - 4}((2, r - 1), \, 7 - r)^\st) + \eulertop(\Hilb(0, (2, r - 1), 4)) - \eulertop(H_2^{}).
\end{align*}
Thus,
\begin{align*}
\eulertop (\fiberL & Y_{3, r - 1}^{}) - \eulertop(\fiberR Y_{3, r - 1}^{}) \\
& = -6 \eulertop(\M^{4 r - 4}((2, r - 1), \, 7 - r)^\st) + \eulertop(\Hilb(0, (2, r - 1), 4)) - \eulertop(H_2^{}) \\
& = -6 \eulertop(\Hilb((2, r - 1), 4)) + 60 \eulertop(\Hilb((2, r - 2), 0)) + \eulertop(\Hilb(0, (2, r - 1), 4)) - 60 (r - 2).
\end{align*}
We have
\[
Y_{3, r - 2}^{} = \M^{4 r - 4}((2, r - 2), \, 4 - r)^\st \times \M((0, 2), 6).
\]
In view equation~\eqref{r+t=2}, we have the isomorphisms
\[
\M^{4 r - 4}((2, r - 2), \, 4 - r)^\st \simeq \M^\infty((2, r - 2), \, 4 - r) \simeq \left| \Osh(2, r - 2) \right|.
\]
Consider $\Lambda' = (\H^0(\Osh_C^{}), \Osh_C^{}) \in \M^\infty((2, r - 2), \, 4 - r)$.
Consider distinct lines $L_1^{}$ and $L_2^{} \subset X$ of degree $(0, 1)$.
According to Theorem~\ref{euler_fibers}(ii) and Proposition~\ref{Xi_calculation}(ii), we have the equation
\begin{multline*}
\eulertop(\fiberL (\langle \Lambda' \rangle, \langle \Osh_{L_1}^{}(2, 0) \oplus \Osh_{L_2}^{}(2, 0) \rangle))
= \Xi_\Lt^{}(\Lambda', L_1^{}, 2, \{ 1 \}) \, \Xi_\Lt^{}(\Lambda', L_2^{}, 2, \{ 1 \}) \\
= (\hom(\Osh_C^{}, \Osh_{L_1}^{}) + 2) (\hom(\Osh_C^{}, \Osh_{L_2}^{}) + 2)
= \begin{cases}
4 & \text{if $L_1^{} \nsubseteq C$ and $L_2^{} \nsubseteq C$}, \\
6 & \text{if $L_1^{} \subset C$ or $L_2^{} \subset C$, but $L_1^{} \cup L_2^{} \nsubseteq C$}, \\
9 & \text{if $L_1^{} \cup L_2^{} \subset C$}.
\end{cases}
\end{multline*}
Consider a line $L \subset X$ of degree $(0, 1)$.
According to loc.cit., we have the equation
\begin{multline*}
\eulertop(\fiberL (\langle \Lambda' \rangle, \langle 2 \Osh_L^{}(2, 0) \rangle))
= \Xi_\Lt^{}(\Lambda', L, 2, \{ 1, 1 \}) + \Xi_\Lt^{}(\Lambda', L, 2, \{ 2 \}) \\
= \binom{\hom(\Osh_C^{}, \Osh_L^{}) + 2}{2} + (\hom(\Osh_C^{}, \Osh_{2 L}^{}) - \hom(\Osh_C^{}, \Osh_L^{}) + 2)
= \begin{cases}
3 & \text{if $L \nsubseteq C$}, \\
5 & \text{if $L \subset C$ but $2 L \nsubseteq C$}, \\
6 & \text{if $2 L \subset C$}.
\end{cases}
\end{multline*}
Recall the isomorphism $\vartheta \colon \left| \Osh(0, 2) \right| \to \M((0, 2), 6)$ from Proposition~\ref{vartheta}.
For $1 \le i \le 6$, consider the locally closed subvarieties $S_i^{} \subset Y_{3, r - 2}^{}$ defined as follows:
\begin{align*}
S_1^{} & = \id \times \vartheta \{ (C, L_1^{} + L_2^{}) \mid L_1^{} \nsubseteq C, \ L_2^{} \nsubseteq C \}, \\
S_2^{} & = \id \times \vartheta \{ (C, L_1^{} + L_2^{}) \mid L_1^{} \subset C \ \text{or} \ L_2^{} \subset C, \ L_1^{} \cup L_2^{} \nsubseteq C \}, \\
S_3^{} & = \id \times \vartheta \{ (C, L_1^{} + L_2^{}) \mid L_1^{} \cup L_2^{} \subset C \}, \\
S_4^{} & = \id \times \vartheta \{ (C, 2 L) \mid L \nsubseteq C \}, \\
S_5^{} & = \id \times \vartheta \{ (C, 2 L) \mid L \subset C, \ 2 L \nsubseteq C \}, \\
S_6^{} & = \id \times \vartheta \{ (C, 2 L) \mid 2 L \subset C \}.
\end{align*}
They have Euler characteristic $0$, $6$, $3 r - 9$, $6$, $6$, respectively, $6 r - 18$.
Applying Lemma~\ref{multiplicative} and the additivity of the Euler characteristic, we obtain the equation
\[
\eulertop(\fiberL Y_{3, r - 2}^{})
= 4 \eulertop(S_1^{}) + 6 \eulertop(S_2^{}) + 9 \eulertop(S_3^{}) + 3 \eulertop(S_4^{}) + 5 \eulertop(S_5^{}) + 6 \eulertop(S_6^{})
= 63 r - 105.
\]
According to Theorem~\ref{euler_loci}(i), we have the equation
\[
\eulertop(\fiberR Y_{3, r - 2}^{}) = \binom{11}{2} \eulertop(\left| \Osh(2, r - 2) \right|) = 165 r - 165.
\]
Combining the last two equations yields the formula
\[
\eulertop(\fiberL Y_{3, r - 2}^{}) - \eulertop(\fiberR Y_{3, r - 2}^{}) = -102 r + 60.
\]
Applying Proposition~\ref{0+}, we obtain the equation
\begin{align*}
\eulertop(\M^{0+}((2, r), \, 10 - r)) & = \eulertop(\Hilb((2, r), 8)) \\
& \phantom{{} =} {} - \phantom{1} 14 \eulertop(\Hilb((2, r - 1), 0)) + \phantom{1} 5 \eulertop(\Hilb(0, (2, r - 1), 0)) \\
& \phantom{{} =} {} - \phantom{1} 12 \eulertop(\Hilb((2, r - 1), 1)) + \phantom{1} 4 \eulertop(\Hilb(0, (2, r - 1), 1)) \\
& \phantom{{} =} {} - \phantom{1} 10 \eulertop(\Hilb((2, r - 1), 2)) + \phantom{1} 3 \eulertop(\Hilb(1, (2, r - 1), 1)) \\
& \phantom{{} =} {} - \phantom{11} 8 \eulertop(\Hilb((2, r - 1), 3)) + \phantom{1} 3 \eulertop(\Hilb(0, (2, r - 1), 2)) \\
& \phantom{{} =} {} - \phantom{11} 6 \eulertop(\Hilb((2, r - 1), 4)) + \phantom{1} 2 \eulertop(\Hilb(1, (2, r - 1), 2)) \\
& \phantom{{} =} {} - \phantom{11} 4 \eulertop(\Hilb((2, r - 1), 5)) + \phantom{11} \eulertop(\Hilb(2, (2, r - 1), 2)) \\
& \phantom{{} =} {} - \phantom{11} 2 \eulertop(\Hilb((2, r - 1), 6)) + \phantom{1} 2 \eulertop(\Hilb(0, (2, r - 1), 3)) \\
& \phantom{{} =} \phantom{{} - 112 \eulertop(\Hilb((2, r - 1), 6))} + \phantom{11} \eulertop(\Hilb(1, (2, r - 1), 3)) \\
& \phantom{{} =} \phantom{{} - 112 \eulertop(\Hilb((2, r - 1), 6))} + \phantom{11} \eulertop(\Hilb(0, (2, r - 1), 4)) \\
& \phantom{{} =} {} + 112 \eulertop(\Hilb((2, r - 2), 0)) - 14 \eulertop(\Hilb(0, (2, r - 2), 0)) \\
& \phantom{{} =} {} + \phantom{1} 40 \eulertop(\Hilb((2, r - 2), 1)) - \phantom{1} 8 \eulertop(\Hilb(0, (2, r - 2), 1)) \\
& \phantom{{} =} {} + \phantom{1} 18 \eulertop(\Hilb((2, r - 2), 2)) - \phantom{1} 2 \eulertop(\Hilb(1, (2, r - 2), 1)) \\
& \phantom{{} =} {} + \phantom{11} 8 \eulertop(\Hilb((2, r - 2), 3)) - \phantom{1} 2 \eulertop(\Hilb(0, (2, r - 2), 2)) \\
& \phantom{{} =} {} + \phantom{111} \eulertop(\Hilb((2, r - 2), 4)) \\
& \phantom{{} =} {} - \phantom{1} 18 \eulertop(\Hilb((2, r - 3), 0)) + \phantom{11} \eulertop(\Hilb(0, (2, r - 3), 0)) \\
& \phantom{{} =} {} - \phantom{11} 4 \eulertop(\Hilb((2, r - 3), 1)) \\
& \phantom{{} =} {} - 162 r + 180.
\end{align*}

\subsection{The final calculation.}
\label{the_end}
According to \cite[Theorem 6.4]{hilbert}, we have the following formulas:
\begin{alignat*}{2}
\eulertop(\Hilb((2, r), 1)) & = \phantom{00}12 r + \phantom{0000}8 \qquad && \text{for $r \ge 0$}, \\
\eulertop(\Hilb((2, r), 2)) & = \phantom{00}42 r + \phantom{000}14 && \text{for $r \ge 1$}, \\
\eulertop(\Hilb((2, r), 3)) & = \phantom{0}120 r && \text{for $r \ge 2$}. \\
\intertext{According to \cite[Table 4]{hilbert}, we have the following formulas:}
\eulertop(\Hilb((2, r), 4)) & = \phantom{0}315 r - \phantom{000}95 && \text{for $r \ge 3$}, \\
\eulertop(\Hilb((2, r), 5)) & = \phantom{0}756 r - \phantom{00}444 && \text{for $r \ge 4$}, \\
\eulertop(\Hilb((2, r), 6)) & = 1722 r - \phantom{0}1476 && \text{for $r \ge 5$}, \\
\eulertop(\Hilb((2, r), 7)) & = 3720 r - \phantom{0}4148 && \text{for $r \ge 6$}, \\
\eulertop(\Hilb((2, r), 8)) & = 7740 r - 10542 && \text{for $r \ge 7$}.
\end{alignat*}
According to Remark~\ref{k=0} and Propositions~\ref{k=l}, \ref{k=1,l=2} and \ref{k=1,l=3}, we have the following formulas:
\begin{align*}
\eulertop(\Hilb(0, (2, r), l)) & = 2 \eulertop(\Hilb((2, r - 1), l)) \qquad \text{for $r \ge 1$}, \\
\eulertop(\Hilb(1, (2, r), 1)) & = \phantom{0}4 \qquad \text{for $r \ge 1$}, \\
\eulertop(\Hilb(2, (2, r), 2)) & = \phantom{0}4 \qquad \text{for $r \ge 1$}, \\
\eulertop(\Hilb(1, (2, r), 2)) & = 20 \qquad \text{for $r \ge 2$}, \\
\eulertop(\Hilb(1, (2, r), 3)) & = 76 \qquad \text{for $r \ge 3$}.
\end{align*}
Substituting the above formulas into the equations found in sections~\ref{c<10} and \ref{c=10}, we obtain the following expressions:
\begin{alignat}{2}
\label{r+t=3}
\eulertop(\M^{0+}((2, r), \, 3 - r)) & = \phantom{00}12 r + \phantom{000}8 \qquad && \text{for $r \ge 2$}, \\
\label{r+t=4}
\eulertop(\M^{0+}((2, r), \, 4 - r)) & = \phantom{00}36 r + \phantom{00}14 \qquad && \text{for $r \ge 3$}, \\
\label{r+t=5}
\eulertop(\M^{0+}((2, r), \, 5 - r)) & = \phantom{00}84 r + \phantom{000}8 \qquad && \text{for $r \ge 4$}, \\
\label{r+t=6}
\eulertop(\M^{0+}((2, r), \, 6 - r)) & = \phantom{0}174 r - \phantom{00}32 \qquad && \text{for $r \ge 5$}, \\
\label{r+t=7}
\eulertop(\M^{0+}((2, r), \, 7 - r)) & = \phantom{0}324 r - \phantom{0}152 \qquad && \text{for $r \ge 6$}, \\
\label{r+t=8}
\eulertop(\M^{0+}((2, r), \, 8 - r)) & = \phantom{0}564 r - \phantom{0}422 \qquad && \text{for $r \ge 7$}, \\
\label{r+t=9}
\eulertop(\M^{0+}((2, r), \, 9 - r)) & = \phantom{0}924 r - \phantom{0}952 \qquad && \text{for $r \ge 8$}, \\
\label{r+t=10}
\eulertop(\M^{0+}((2, r), \, 10 - r)) & = 1449 r - 1897 \qquad && \text{for $r \ge 9$}.
\end{alignat}
By the argument of \cite[Proposition 4.2.9]{choi_thesis}, one can show
that, for all $r$ and $s$, we have the equation
\[
\eulertop(\M((s, r), 1)) = \eulertop(\M^{0+}((s, r), 1)) - \eulertop(\M^{0+}((s, r), -1)).
\]
This is a direct consequence of the duality isomorphism $\M((s, r), 1) \simeq \M((s, r), -1)$ from \cite[Theorem 5]{dedicata_2}
and of the isomorphism~\eqref{phi}.
At equations~\eqref{r+t=2} and \eqref{r+t=3}--\eqref{r+t=10},
we have calculated the r.h.s.\ for $s = 2$ and $2 \le r \le 9$.
We obtain the following expressions:
\begin{align}
\label{r=2}
\eulertop(\M((2, 2), 1)) & = \phantom{00}32, \\
\label{r=3}
\eulertop(\M((2, 3), 1)) & = \phantom{0}110, \\
\label{r=4}
\eulertop(\M((2, 4), 1)) & = \phantom{0}288, \\
\label{r=5}
\eulertop(\M((2, 5), 1)) & = \phantom{0}644, \\
\label{r=6}
\eulertop(\M((2, 6), 1)) & = 1280, \\
\label{r=7}
\eulertop(\M((2, 7), 1)) & = 2340, \\
\label{r=8}
\eulertop(\M((2, 8), 1)) & = 4000, \\
\label{r=9}
\eulertop(\M((2, 9), 1)) & = 6490.
\end{align}
Equation~\eqref{r=2} is also a direct consequence of \cite[Proposition 12]{ballico_huh}.
Equations~\eqref{r=3} and \eqref{r=4} have also been obtained at \cite[Theorem 3]{dedicata_2}, respectively, at \cite[Theorem 1.2]{advances}.

\end{document}